 \documentclass[reqno,11pt]{amsart}
\usepackage{amsmath, amsthm, a4, latexsym, amssymb}
\usepackage[unicode]{hyperref} 
\usepackage{color}

\usepackage{mathrsfs}

\def\@rmrk#1#2{\refstepcounter
    {#1}\@ifnextchar[{\@yrmrk{#1}{#2}}{\@xrmrk{#1}{#2}}}

\makeatletter\@addtoreset{equation}{section}\makeatother

 \newfont{\bfit}{cmbxti10 scaled 1200}

\renewcommand{\d}{{\rm d}}
 \newcommand{\e}{{\rm e} }

 \newcommand{\eps}{\varepsilon}

 \newcommand{\R}{\mathbb{R}}
 \newcommand{\N}{\mathbb{N}}
 \newcommand{\Z}{\mathbb{Z}}

 \newcommand{\E}{\mathbf{E}}
 \renewcommand{\P}{\mathbf{P}}
 \newcommand{\bE}{\mathbb{E}}
 \newcommand{\bP}{\mathbb{P}}
 \newcommand{\bQ}{\mathbb{Q}}
 
 \newcommand{\hP}{\widehat{\mathscr {M}}}
 \def\1{{\mathchoice {1\mskip-4mu\mathrm l} 
{1\mskip-4mu\mathrm l}
{1\mskip-4.5mu\mathrm l} {1\mskip-5mu\mathrm l}}}

 \newcommand{\Mcal}{{\mathcal M}}

\newcommand{\weak}{{\Rightarrow}}

\newcommand{\X}{{\widetilde{\mathcal{X}}} }

\newcommand{\ssup}[1] {{\scriptscriptstyle{({#1}})}}

\renewcommand{\subsection}{\secdef \subsct\sbsect}
\newcommand{\subsct}[2][default]{\refstepcounter{subsection}
\vspace{0.15cm}
{\flushleft\bf \arabic{section}.\arabic{subsection}~\bf #1  }
\nopagebreak\nopagebreak}
\newcommand{\sbsect}[1]{\vspace{0.1cm}\noindent
{\bf #1}\vspace{0.1cm}}

{\nopagebreak {\hfill\rule{2mm}{2mm}}\\ }

\newtheorem{theorem}{Theorem}[section]
\newtheorem{lemma}[theorem]{Lemma}

\newtheorem{cor}[theorem]{Corollary}
\newtheorem{prop}[theorem]{Proposition}

\newtheoremstyle{thm}{1.5ex}{1.5ex}{\itshape\rmfamily}{}
{\bfseries\rmfamily}{}{2ex}{}

\newtheoremstyle{rem}{1.3ex}{1.3ex}{\rmfamily}{}
{\itshape\rmfamily}{}{1.5ex}{}
\theoremstyle{rem}
\newtheorem{remark}{{\slshape\sffamily Remark}}[]

\refstepcounter{subsubsection}

\def\thebibliography#1{\section*{References}
  \list%
  {\arabic{enumi}.}
    {\settowidth\labelwidth{[#1]}\leftmargin\labelwidth
    \advance\leftmargin\labelsep
    \parsep0pt\itemsep0pt
    \usecounter{enumi}}
    \def\newblock{\hskip .11em plus .33em minus .07em}
    \sloppy                   
    \sfcode`\.=1000\relax}

\begin{document}
\title[Localization of the Gaussian multiplicative chaos in the Wiener space]
{\large Localization of the Gaussian multiplicative chaos in the Wiener space and the stochastic heat equation in strong disorder}
\author[Yannic Br\"oker and Chiranjib Mukherjee]{}
\maketitle
\vspace{-0.5cm}

\centerline{\sc Yannic Br\"oker\footnote{University of M\"unster, Einsteinstrasse 62, M\"unster 48149, Germany, {\tt yannic.broeker@uni-muenster.de}} and 
 Chiranjib Mukherjee\footnote{University of M\"unster, Einsteinstrasse 62, M\"unster 48149, Germany,  {\tt chiranjib.mukherjee@uni-muenster.de}}}


\renewcommand{\thefootnote}{}
\footnote{\textit{AMS Subject
Classification:} 60K35, 60J65, 60J55, 60F10, 35R60, 35Q82, 60H15, 82D60}
\footnote{\textit{Keywords:} Gaussian multiplicative chaos, supercritical, renormalization, glassy phase, freezing, stochastic heat equation, strong disorder, asymptotic pure atomicity, translation-invariant compactification} 

\vspace{-0.5cm}
\centerline{\textit{University of M\"unster}}
\vspace{0.2cm}

\begin{center}
\today
\end{center}

\renewcommand{\thefootnote}{\fnsymbol{footnote}}

\begin{quote}{\small {\bf Abstract: }
We consider 
a {\it{Gaussian multiplicative chaos}} (GMC) measure on the classical Wiener space driven by a smoothened (Gaussian) space-time white noise. 
For $d\geq 3$ it was shown in \cite{MSZ16} that for small noise intensity, the total mass of the GMC converges to a strictly positive random variable, while larger disorder strength (i.e., low temperature) forces 
the total mass to lose uniform integrability, eventually producing a vanishing limit. 
Inspired by strong localization phenomena for log-correlated Gaussian fields and Gaussian multiplicative chaos in the finite dimensional Euclidean spaces (\cite{MRV16,BL18}),
and related results for discrete directed polymers (\cite{V07,BC16}), 
we study the endpoint distribution of a Brownian path under the {\it{renormalized}} GMC measure in this setting.
We show that in the low temperature regime, the energy landscape of the system freezes and enters the so called {\it{glassy phase}} as the entire mass of the Ces\`aro average of the endpoint GMC distribution
stays localized in few spatial islands, forcing the endpoint GMC to be {\it{asymptotically purely atomic}} (\cite{V07}). 
The method of our proof is based on the translation-invariant compactification introduced in \cite{MV14} and 
a fixed point approach related to the cavity method from spin glasses recently used in \cite{BC16} in the context of the directed polymer model in the lattice.}
\end{quote}

\section{Introduction and the main result.}

\subsection{Motivation.}

Let $\Omega$ be a metric space which is endowed with a finite measure $\mu$. Consider the tilted random measure of the form 
\begin{equation}\label{GMC}
\mathscr M_\beta(\d\omega)= \mathscr M_{\beta,\mathscr H}(\d\omega)= \exp\big\{\beta \mathscr H(\omega)- \frac 12 \beta^2 \E[\mathscr H(\omega)^2]\big\} \,\, \mu(\d\omega)
\end{equation}
where $\beta>0$ is a parameter and $\{\mathscr H(\omega)\}_{\omega\in \Omega}$ is a centered Gaussian field defined on a complete probability space $(\mathcal E, \mathcal F, \P)$. 
The theory of {\it{Gaussian multiplicative chaos}} (GMC), whose idea was first propounded by Kahane (\cite{K85}), is the generalization of \eqref{GMC} to the setting when the 
random field $\{\mathscr H(\omega)\}$ lives on the space of distributions, that is, they are defined via a family of integrals w.r.t. a suitable class of test functions.

One of the crucial properties of the GMC is captured by the following simple comparison principle which was also discovered by Kahane (\cite{K85}). If $\{\mathscr H(\omega)\}$ and $\{\mathscr G(\omega)\}$ are two continuous Gaussian fields such that $\E[\mathscr G (\omega_1) \mathscr G (\omega_2)] \leq \E[\mathscr H(\omega_1)\, \mathscr H(\omega_2)]$ for all $\omega_1,\omega_2\in \Omega$, then for any concave function $F:\R_+\to \R$ with at most polynomial growth at infinity, 
\begin{equation}\label{comparison}
\E[F(\mathscr Z_{\beta,\mathscr G})] \geq \E[F(\mathscr Z_{\beta,\mathscr H})] \qquad \mbox{where}\quad \mathscr Z_{\beta,\cdot}=\int_\Omega \mathscr M_{\beta,\cdot}(\d\omega).
\end{equation} 

In the finite dimensional setting, GMC measures share close connection to the two-dimensional Liouville quantum gravity (\cite{DS11}) and 
its studies have seen a lot of revived  interest in the recent years. In this setting, the relevant measures are defined as $ M_{\beta,T}(\d x): = \e^{-\beta^2 T/2} \e^{\beta X_T(x)} \d x$ where $D$ is a domain in $\R^d$, $\d x$ stands for the Lebesgue measure 
and the ambient Gaussian field $(X_T(x))_{x\in D}$ is {\it{log-correlated or star-scale invariant}} after a suitable cut-off regularization at level $T$. A rigorous construction of the limiting 
measure $\lim_{T\to\infty} M_{\beta,T}$ has been carried out using a martingale approximation (\cite{K85})
and it is well known that when $\beta < \sqrt{2d}$, 
$M_{\beta,T}$ converges as $T\to\infty$ toward a nontrivial measure $M_\beta$ which is diffuse and is known as the {\it{subcritical GMC}}, while for $\beta\geq\sqrt{2d}$, $M_{\beta,T}$ converges to $0$ as $T\to\infty$. 
In this setting (i.e., for log-correlated fields in $\R^d$), a rigorous construction of the subcritical GMC measure also follows from a stable mollification procedure (see \cite{RV10,DS11,B17}). 
Alternatively, a subcritical GMC in a general setting is also characterized by requiring that 
$\mathscr M_{\beta,\mathscr H+ v}(\d\omega)=\e^{v(\omega)} \, \mathscr M_{\beta,\mathscr H}(\d\omega)$
for every Cameron-Martin vector $v$ for the Gaussian field $\mathscr H$, that is,  for all deterministic $v:\Omega\to \R$ such that the law of $\mathscr H+v$ is absolutely continuous w.r.t. that of $\mathscr H$ (see \cite{S14}).

In the finite dimensional setting, the regime $\beta> \sqrt{2d}$ corresponds to the so-called {\it{supercritical phase}} of the GMC measures, which 
has also received much attention in the physics literature (see \cite{M74,DS88} for questions on dyadic trees and \cite{CLD01,FB08,FLDR09} for log-correlated fields). Heuristically speaking, 
in this regime, one expects the energy landscape of the underlying Gaussian field to {\it{freeze}} and enter a {\it{glassy phase}}. On a rigorous level, 
for log-correlated or star-scale invariant Gaussian fields in the Euclidean set up, this has been justified rigorously in \cite{MRV16} (see also \cite{BL18} for similar
 results for discrete $2d$ Gaussian free field). In particular, it was  shown that for $\beta>\sqrt{2d}$ and 
for suitable constants $\lambda_1(\beta),\lambda_2(\beta)>0$, the {\it{renormalized}} GMC measure 
$$
\e^{\lambda_1(\beta) \log t+ \lambda_2(\beta) t} M_{\beta,t}
$$
in the limit $t\to\infty$ is supported only on atoms.\footnote[7]{In fact, in the literature cited above, it is shown that, for $\beta>\sqrt{2d}$, the GMC measure $M_{\beta,t}$ concentrates its mass 
only on sites close to {\it{centered maximum}} $\sup_{x\in D} [X_t(x)- \sqrt{2d} t]$ of the field 
and consequently, the limiting measure is described as a Poisson measure with (random) intensity given by the {\it{derivative martingale}} or the {\it{critical GMC}} at $\beta=\sqrt{2d}$ 
whose construction was rigorously carried out in (\cite{DRSV14-I,DRSV14-II}).} 

Quite naturally, the above results inspire questions concerning the behavior of supercritical GMC in the {\it{infinite dimensional setting}}, which have not been explored to the best of our knowledge. 
In the present context, we drop all assumptions regarding log-correlations or star-scale invariance of the underlying field and 
consider a GMC measure of the form \eqref{GMC} on a noncompact metric space. In this setting we show that
 when the temperature is low, the limiting measures are also supported only on atoms as the cut-off level is sent off to infinity. 
We now turn to a precise mathematical layout of the problem. 

We fix any spatial dimension $d \geq 1$ and set $\Omega= C([0,\infty);\R^d)$  to be the metric space of continuous functions endowed with the topology of uniform convergence of compact subsets.
$\Omega$ is tacitly equipped with the Wiener measure $\bP_x$ corresponding to a $\R^d$-valued Brownian motion starting at $x\in \R^d$.
We denote by $B$ a cylindrical Wiener process and by $\dot B$ a Gaussian space-time white noise which is independent of the path $W$. In other words, for any Schwartz function $\varphi\in \mathcal S(\R_+\times \R^d)$, $\dot B(\varphi)$ is a Gaussian random variable on a fixed  probability space $(\mathcal E,\mathcal F, \mathbf P)$ 
with mean $0$ and covariance $\mathbf E[ \dot B(\varphi_1)\,\, \dot B(\varphi_2)]= \int_0^\infty \int_{\R^d} \varphi_1(t,x) \varphi_2(t,x) \d x \d t$. Throughout the rest of the article,
 $\E$ will denote expectation w.r.t. $\P$. We also fix a nonnegative function $\phi$ which is smooth, spherically symmetric and is supported in a ball $B_{1/2}(0)$ of radius $1/2$ around $0$ and normalized to have total mass $\int_{\R^d} \phi(x)\d x=1$. Then we have a (spatially convolved white noise) Gaussian field $\{\mathscr H_T(W)\}_{W\in\Omega}$ at level $T$, defined as
\begin{equation}\label{mathscr-H}
\mathscr H_T(W)= \mathscr H_T(W,\dot B)=\int_0^T \int_{\R^d} \, \phi(W_{s}-y) \, \dot B(s, y)\, \d y\,\, \d s. 
\end{equation}
The corresponding tilted measure 
\begin{equation} \label{eq:M}
\begin{aligned}
&\mathscr M_{\beta,T}(\d W)=  \exp\left\{\beta \mathscr H_T(W) - \frac{\beta^2T}{2}(\phi\star\phi)(0)\right\} \bP_0(\d W)
\end{aligned}
\end{equation}
is then readily interpreted as a Gaussian multiplicative chaos indexed by Wiener paths (recall \eqref{GMC}). It has covariance kernel 
\begin{equation}\label{covariance-comparison}
\begin{aligned}
\E\big[\mathscr H_T(W^{\ssup 1})\,\mathscr H_T(W^{\ssup 2})\big]
&= \int_0^T\d s \int_{\R^d} \d y\,\,\phi(W_s^{\ssup 1}-y)\,\phi(W^{\ssup 2}_s-y)\\
&=\int_0^T \d s\,\,(\phi\star\phi)(W^{\ssup 1}_s-W^{\ssup 2}_s) \leq T (\phi\star \phi)(0).
\end{aligned}
\end{equation} 
If we denote the total mass by $\mathscr Z_{\beta,T} = \int_\Omega \mathscr M_{\beta,T}(\d W)$, we also have the renormalized GMC measure 
\begin{equation}\label{polymer-measure}
\begin{aligned}
\hP_{\beta,T}(\d W)&= \frac 1 {\mathscr Z_{\beta,T}} \mathscr M_{\beta,T}(\d W).
\end{aligned}
\end{equation}
Using Kahane's comparison inequality (recall \eqref{comparison}) and the domination of the kernels \eqref{covariance-comparison}, it was also shown
 in (\cite{MSZ16}) that, in $d\geq 3$, there exists $\beta_c\in (0,\infty)$ so that for $\beta<\beta_c$, the total mass $\mathscr Z_{\beta,T}=\int_\Omega \mathscr M_{\beta,T}(\d W)$ of the GMC 
converges in probability to a strictly positive random variable, while for $\beta>\beta_c$, $\mathscr Z_{\beta,T}$ ceases to be uniformly integrable and eventually collapses to zero as $T\to\infty$. 
In the present context, the main result of our article shows that, loosely speaking, 
when the temperature is sufficiently low, in particular when $\lim_{T\to\infty} \mathscr Z_{\beta,T}=0$,
the renormalized GMC measure $\hP_{\beta,T}(\d W)=  {\mathscr Z_{\beta,T}}^{-1} \mathscr M_{\beta,T}(\d W)$ has no asymptotic disintegration of mass--its {\it{entire mass}} is preserved and 
accumulated in few randomly located islands in $\R^d$.
Given the above discussion pertaining to low-temperature localization of GMC in finite dimensions (e.g., \cite{MRV16}),
the present result is then a contribution toward a rigorous understanding of 
{\it{atomic (or supercritical)}} GMC  in the infinite dimensional setting. We turn to a precise statement of our main result. 

  \subsection{The result.} 
 
 We set 
  \begin{equation}\label{beta_1}
 \Lambda(\beta)=- \lim_{T\to\infty} \frac 1 T \E[\log \mathscr Z_{\beta,T}]
 \end{equation}
where $\mathscr Z_{\beta,T}=\int \mathscr M_{\beta,T}(\d W)$. It is easy to see via a subadditivity argument that the above limit always exists and Jensen's inequality together with the fact that $\E[\exp\{\beta \mathscr H_T(W,B)\}]=\exp\{\frac{\beta^2T}{2}(\phi\star\phi)(0)\}$ forces it to be nonnegative. Furthermore, 
the map $\beta\mapsto \Lambda(\beta)$ is monotone increasing and $\Lambda(\beta)>0$ also implies that $\lim_{T\to\infty} \mathscr Z_{\beta,T}=0$ almost surely; see Theorem \ref{Lyapunov}. 
For any $\eps, t >0$, we define the regions 
\begin{equation}\label{eq:U}
U_{t,\eps}= \{x\in \R^d\colon \bQ_{\beta,t}[B_1(x)] > c_0 \eps \} \qquad c_0=|B_1(0)|
\end{equation}
that carry uniformly positive density for the GMC endpoint 
\begin{equation}\label{polymer-endpoint}
\bQ_{\beta,t}= \hP_{\beta,t} \,\, W_t^{-1}. 
\end{equation}

Here is our first main result. 
 
 \begin{theorem}[Pure atomicity of the GMC endpoint]\label{APA}
	Let $d\geq 1$ and fix $\beta>\beta_1:=\inf\{\beta>0\colon \Lambda(\beta)>0\} \in [0,\infty]$. 
	Then for any sequence $\eps_t\to 0$ with $t\rightarrow\infty$, 
	\begin{align}\label{eq-APA}
		\lim_{T\rightarrow\infty}\frac{1}{T}\int_0^T \bQ_{\beta,t}[U_{t,\eps_t}]\,\, \d t=1\qquad\P-\mbox{a.s.}
	\end{align}
\end{theorem}

Note that Theorem \ref{APA} holds for any $d\geq 1$ as long as  $\Lambda(\beta)>0$. Furthermore, Theorem \ref{Lyapunov} shows that $\Lambda(\beta)>0$ implies $\lim_{T\to\infty} \mathscr Z_{\beta,T}=0$, 
and combined with Theorem \ref{thm 4.9.} it also shows that the latter convergence is in fact {\it{exponential}} $\mathscr Z_{\beta,T}=\e^{-T[\Lambda(\beta)+o(T)]}$, 
which contrasts the polynomial rate of the convergence of the (Gaussian) fluctuations of $\mathscr Z_{\beta,T}$ when $\beta$ is small; see \cite{CCM18}.

We refer to the interesting works (\cite{CSY03,V07,BC16,Ba17,C18,CC18}) where low temperature localization properties of discrete directed 
polymers have been extensively studied. In the lattice setting, in 
\cite{CSY03} the averages on the left-hand side in \eqref{eq-APA} were shown to be uniformly bounded below by a constant $c\in (0,1]$. The latter statement was later strengthened in 
\cite{V07} for heavy-tailed environments (i.e., when the logarithmic moment generating function
is infinity). Very recently, substantial progress was made when the latter statement was shown to be true in \cite{BC16} for polymers in the lattice even 
with finite exponential moments. We also remark that 
localization properties for polymers in the lattice setting can be efficiently studied by using the method of fractional moments introduced in (\cite{CSY03}).
In the continuous setting,  this method, however, seems to break down, and particularly 
 for Gaussian fields, techniques from GMC like comparison inequalities are well suited and efficient, as demonstrated in \cite{MSZ16}. 
 We refer to Section \ref{sec-comparison-proof} for a comparison of techniques of the proofs.

We mention that the GMC \eqref{eq:M} is also closely related to the multiplicative noise {\it{stochastic heat equation}} which is formally written as 
\begin{equation}\label{she-formal}
\d u_t= \frac 12 \Delta u_{t} \d t + \beta   u_{t}\,  \d B_{t}. 
\end{equation}
Although equation \eqref{she-formal} is a priori ill-posed, when $d=1$ substantial recent progress has been made in giving a rigorous meaning to its solution (\cite{BC95,BG97,SS10,ACQ11,H13,AKQ,GP17}, see also \cite{BC98,CSZ17} for the case $d=2$). It is natural to consider a regularized version 
\begin{equation}\label{she-intro}
\d u_{\eps,t} =\frac 12 \Delta u_{\eps,t} \d t + \beta(\eps,d)\,\,    u_{\eps,t} \d B_{\eps,t} \;,\qquad\,\,   u_{\eps,0}(x) =1,
\end{equation}
of \eqref{she-formal} by interpreting the above stochastic differential in the classical It\^o sense and considering the spatially mollified noise
$B_{\eps,t}(x) = \dot B \left( \varphi_{\eps,t,x}\right)$ with $\varphi_{\eps,t,x} (s,y) =  {\1}_{[0,t]}(s) \phi_\eps(y - x)$
and  $\phi_\eps(\cdot)=\eps^{-d}\,\phi(\frac{\cdot}\eps)$ being an approximation of the Dirac-delta. 
Clearly $B_{\eps,t}(x)$ is again a centered Gaussian process with covariance $\E[ B_{\eps,t}(x) B_{\eps,s}(y) ] = (s\wedge t) \big(\phi_\eps\star \phi_\eps\big)(x-y) = (s\wedge t) \, \eps^{-d} V((x-y)/\eps)$, where
\begin{equation}\label{eq:V}
V=\phi \star \phi
\end{equation}
is a smooth function supported in the unit ball $B_1(0)$ around the origin. Then
\begin{equation}\label{she-FK}
\begin{aligned}
u_{\eps,t}(x)
&= \mathbb E_x \bigg[ \exp\bigg\{ \beta(\eps,d) \int_0^t \int_{\R^d} \, \phi_\eps(W_{t-s}-y) \, \dot B(s, y)\, \d y\,\, \d s -  \, \frac {t\beta(\eps,d)^2} 2 \,\eps^{-d}\, V(0)\bigg\}\bigg]
\end{aligned}
\end{equation}
provides the renormalized Feynman-Kac solution to \eqref{she-intro}, and for $d\geq 3$ if we choose
$$
\beta(\eps,d)=\beta \eps^{\frac{d-2}2} \qquad\mbox{and}\quad d\geq 3,\qquad \beta>0,
$$
 then by Brownian scaling and time-reversal, 
\begin{equation}\label{scaling}
u_{\eps,t} (\cdot)\,\, \overset{\ssup d}= \,\,\mathscr Z_{\beta,\eps^{-2}t}(\eps^{-1} \cdot)\;,
\end{equation}
where $\mathscr Z_{\beta,T}(x) = \int_\Omega \mathscr M_{\beta,T}^{\ssup x}(\d W)$ is the total mass of the GMC measure weighted w.r.t. the Wiener measure $\bP_x$. 
When $\beta>0$ is sufficiently small, asymptotic behavior of the solutions (as $\eps\to 0$) as well as associated measures 
have been studied extensively (see \cite{MSZ16,M17,CCM18,BM19,CCM19}). 
When $\beta>0$ is large, then Theorem \ref{APA}, combined with the scaling relation \eqref{scaling} imply the localization effect of the measures associated to \eqref{she-FK} as $\eps\to 0$: 
\begin{cor}[Pure atomicity of the stochastic heat equation]\label{cor-APA}
Let $d\geq 3$ and assume $\Lambda(\beta)>0$ and $\eps_t\to 0$ as $t\to\infty$ as in Theorem \ref{APA}. Then 
$$
\lim_{T\to \infty}\,\, \frac 1 T\int_0^{T}\bar{\mathscr {M}}_{\beta,t^{-1/2}} \, \big[W_t\in U_{t,\eps_t}\big]\mathrm{d}t=1\quad \text{in}\,\,\P-\,\mbox{probability}, 
$$
where $\bar{\mathscr M}_{\beta,\eps}$ is the normalized GMC measure corresponding to the Feynman-Kac solution \eqref{she-FK}.
\end{cor}

\begin{remark}
	In the present setting, the parameter $\beta$ is considered to be a positive (large enough) real number. It is an intriguing question to see if the localization property proved in this article can be extended to a complex GMC in the Wiener space, that is, GMC with complex $\beta$.
\end{remark}

\subsection{Outline of the proof and comparison of proof techniques.}\label{sec-comparison-proof}
In order to provide some guidelines to the reader, we will briefly sketch the central idea of the proof of Theorem \ref{APA} in this section. 
We will also emphasize on the similarities and differences to the earlier approaches used in the existing literature.

As remarked earlier, localization statements for directed polymers were derived using the method of fractional moments (\cite{CSY03,V07}),
while similar results for GMC measures for log-correlated fields in $\R^d$ were proved (\cite{MRV16,BL18}) by studying maximum of branching random walks (\cite{A13,M15}). 
These methods are quite different from the approach used in the present article, for which 
we directly leverage the machinery in \cite{MV14}, while following \cite{BC16}
as a guiding philosophy. 

\subsubsection{Outline of the proof.} The proof of Theorem \ref{APA} splits into three main steps.

\noindent{\it{Step 1:}} The first step is based on studying a metric on the {\it{translation-invariant compactification}} (of the quotient space) of probability measures on $\R^d$ developed in \cite{MV14}.\footnote[3]{Although the compactification 
in \cite{MV14} was carried out for the space $\Mcal_1(\R^d)$ of probability measures on $\R^d$, the exact same construction carries over to the setting of any (Abelian) group acting on the relevant Polish space. In particular, it works also in the lattice setting for the action of $\Z^d$ as an additive group on $\Mcal_1(\Z^d)$.}
Since the method for \cite{MV14} will be a building block of our proof on a conceptual level, it is useful to briefly review its main idea. 

Note that the space $\Mcal_1(\R^d)$  of probability measures on $\R^d$ is noncompact under the usual weak topology determined by convergence of integrals w.r.t. continuous and bounded functions.
There can be several reasons which can be attributed to this phenomenon. For instance, a Gaussian with a very large variance spreads its mass very thin and eventually totally disintegrates into dust. Also, a mixture like $\frac{1}{2}(\mu \star \delta_{a_n}+\mu\star  \delta_{-a_n})$ splits into two (or more) widely separated pieces as $a_n\to \infty$. To compactify this space, we should be allowed to ``center" each piece separately as well as to allow some mass to be ``thinly spread and disappear." The intuitive idea, starting with a sequence of probability distributions $(\mu_n)_n$ in $\R^d$ is to identify a compact region where $\mu_n$ has its largest accumulation of mass.  
By choosing subsequences if necessary, we can assume that for any $r>0$, $\sup_{x\in \R^d} \mu_n\big(B_r(x)\big)\to q(r)$ as $n\to\infty$ and $q(r)\to p_1 \in [0,1]$ as $r\uparrow\infty$. Then there is a shift $\lambda_n=\mu_n\star \delta_{a_n}$ that converges along a subsequence vaguely to a subprobability measure $\alpha_1$ of mass $p_1$. This means $\lambda_n$ can be written as $\alpha_n+\beta_n$ so that $\alpha_n \Rightarrow\alpha_1$ weakly and we recover the partial mass $p_1\in [0,1]$. We peel off $\alpha_n$ from $\lambda_n$ and repeat the same process for $\beta_n$ to get convergence along a further subsequence. We go on recursively to get
convergence of one component at a time along further subsequences in the space of subprobability measures, {\it{modulo spatial shifts}}. The picture is, $\mu_n$ roughly concentrates on 
{\it{widely separated}} compact pieces of masses $\{p_j\}_{j\in\N}$ while the rest of the mass $1-\sum_j p_j$ leaks out.

In other words, given any sequence $\widetilde\mu_n$ of equivalence classes in $\widetilde\Mcal_1(\R^d)$, which is the quotient space of $\Mcal_1(\R^d)$ under spatial shifts, there is a subsequence which converges 
(the convergence criterion is determined by a metric structure, see Section \ref{sec-review-MV14} for the precise definition) to an element $\xi=\{\widetilde\alpha_1,\widetilde\alpha_2,\dots\}$, a collection of equivalence classes of subprobabilities $\alpha_j$ of masses $0\leq p_j\leq 1$, $j\in\N$.\footnote[8]{For example, let $\mu_n$ be the Gaussian mixture $\frac 13 N(n,1)+\frac 13 N(n^2,1)+\frac 1 3 N(0,n)$. 
Then the limiting object for $\widetilde\mu_n$ is the collection $\xi=\{\widetilde\alpha_1,\widetilde\alpha_1\}\in \X$, where $\widetilde\alpha_1$ is the equivalence class of a Gaussian with variance $1$ and weight $\frac 13$.} The space $\X$ of such collections $\xi$ of equivalence classes is the compactification of $\widetilde\Mcal_1(\R^d)$; see Theorem \ref{thm-MV14} below for a precise statement.  
In the present context, then our task boils down to investigating the asymptotic behavior of the {\it{GMC endpoint orbits}} $\widetilde{\bQ}_{\beta,T}$ embedded in $\X$. 

With the function $V=\phi\star \phi$ vanishing at infinity, we heavily exploit the metric structure on the compactification $\X$  to derive continuity properties of shift-invariant functionals 
of the form 
$$
\Phi(\xi)=\frac{\beta^2 V(0)}2- \frac{\beta^2}2 \sum_i \int_{\R^{2d}} V(x-y) \alpha_i(\d x) \alpha_i(\d y) \qquad \xi=(\widetilde\alpha_i)_i\in \X,
$$
on $\X$ (see Sections \ref{subsec-total-mass}-\ref{subsec-Phi}). Methods from stochastic calculus (\cite{CN95,CC13}) then enable us to decompose the polymer free energy $\frac 1 T\log Z_{\beta,T}$ in terms of a martingale and an additive functional of  $\Phi(\widetilde{\bQ}_{\beta,T})$. This step is carried out in Section \ref{sec-free-energy}.

\noindent{\it{Step 2:}} The next main step is to construct a certain dynamics on $\X$ described by transition probabilities $\pi_t(\xi,\d\xi^\prime)=\P[\xi^{\ssup t}\in\d\xi^\prime|\xi]$ with 
$\xi^{\ssup t}=(\widetilde{\alpha}_i^{\ssup t})_{i\in I}$ and $\alpha_i^{\ssup t}\in \Mcal_{\leq 1}$ for any $i\in I$ and $t\geq 0$. Here,
 $\alpha_i^{\ssup t}$ can be seen as the subprobability $\alpha_i$ whose mass gets transported through the space $\R^d$ from time zero to $t$ by the following dynamic: 
$$
\alpha_i^{\ssup t}(\mathrm{d}x):=\frac{1}{\mathscr F_t(\xi)+\E\big[ Z_t-\mathscr F_t(\xi)\big]}\int_{\R^d}\alpha_i (\mathrm{d}z)\bE_z\bigg [\1_{\{W_t\in\mathrm{d}x\}}\, \exp\big\{\beta\mathscr H_t(W)\big\}\bigg] 
$$
where $Z_t=\bE_0[\exp\{\beta\mathscr H_t(W)\}]$, $\mathscr H_t(W)=\int_0^t \int_{\R^d} \, \phi(W_{s}-y) \, \dot B(s, y)\, \d y\,\, \d s$ is the Gaussian field (recall \eqref{mathscr-H}) and 
$$
\mathscr F_t(\xi)=\sum_{i\in I} \int_{\R^d}\int_{\R^d}\alpha_i(\mathrm{d}z) \bE_z\big[\1_{\{W_t\in\mathrm{d}x\}}\,\, \exp\big\{\beta\mathscr H_t(W)\big\}\big].
$$
Section \ref{section time dependence} is then devoted to showing that for any $t>0$, the above kernel map $\xi\mapsto \pi_t(\xi,\cdot)$ is continuous on $\X$. For its proof, we also heavily exploit the precise metric structure 
on the space $\X$. In particular, an important recipe is provided by Proposition \ref{lemma-step5} which is based on a second moment computation that hinges on the notion of total disintegration of mass, an important trait for the topology on $\X$ (see \eqref{decompose-3} for a precise statement) as well as a decoupling phenomenon of two independent GMC chains at large distances that captures the underlying {\it{attractive}} nature of the model. The aforementioned representation of  $\frac 1 T\log Z_{\beta,T}$ and the above continuity of $\xi\mapsto \pi_t(\xi,\cdot)$ 
also imply a variational formula for the (quenched) free energy $\lim_{T\to\infty}\frac 1 T \log Z_{\beta,T}=\inf_{\vartheta \in \mathfrak m} \int \Phi(\xi) \vartheta(\d\xi)$, where the infimum is taken (and given the continuity 
of the above map), attained over the {\it{compact}} set $\mathfrak m=\{\vartheta\in\Mcal_1(\X)\colon \Pi_t(\vartheta,\cdot)=\vartheta \,\forall t\geq 0\}$ of fixed points of $\Pi_t(\vartheta,\cdot)=\int \pi_t(\xi,\cdot)\vartheta(\d\xi)$ for $\vartheta\in \Mcal_1(\X)$.

\noindent{\it{Step 3:}} Finally, one shows that the minimizers $\mathfrak m_0\subset \mathfrak m$ of the above variational formula attract the empirical measures $\frac 1 T\int_0^T \delta_{\widetilde{\bQ}_t} \d t$ of the endpoint orbit and 
as long as $\Lambda(\beta)>0$, no mass dissipates under any $\vartheta\in \mathfrak m_0$, which concludes the proof of Theorem \ref{APA}.

\subsubsection{Comparison with the earlier approach.} As mentioned earlier, we have drawn inspiration from the techniques recently employed in \cite{BC16} for directed polymers 
which also followed the program in \cite{MV14} for constructing a metric on
the compactification in a lattice setting. It was also shown (\cite[Proposition A.3]{BC16}) that the metric therein produces the same topology as \cite{MV14} when the latter structure is adapted to the lattice setting. 
However, the metric in \cite{BC16} is structurally quite different from \cite{MV14}. In particular, the construction of the former metric crucially exploits the countability (graph structure)
of $\Z^d$ and relies on interpreting probability measures on $\Z^d$ as (mass) functions which allows distant point masses to nearly live on separate copies of $\Z^d$.\footnote[9]{In \cite{BC16}, 
the difference $(n,x)-(m,y)$ between any two elements $(n,x), (m,y)\in \N\times \Z^d$ is defined to be infinity if $n\neq m$, while it is $x-y$ if $n=m$.
This interpretation is used in this setting to construct the metric on the set $\mathcal S=\{f\colon \N\times\Z^d \to \R\colon f\geq 0, \sum_{(n,x)} f(n,x)\leq 1\}$ of subpartitioned mass functions and derive its compactness; see \cite[Section 2.1]{BC16} for details.} In this setting, then the rewrite of the polymer free energy is carried out by a telescoping sum and crucial continuity properties of the functionals therein are checked exploiting this distance function between two partitioned mass functions in the lattice setting. 

In contrast (i.e., in absence of the countable lattice structure), in the present context, the crucial continuity properties of the relevant functionals are deduced by directly leveraging the representation structure 
of the metric in \cite{MV14}. Therefore, the actual execution of the machinery in Section \ref{section Space X}-Section \ref{section time dependence} (i.e., for Step 1 and Step 2 in the aforementioned discussion) 
is therefore quite different from the existing literature in the lattice setting. The remaining arguments for the proof of Theorem \ref{APA} are then provided in Section \ref{sec-final-details} by adapting the approach from \cite{BC16}  
to our setting.


\noindent{\it{Organization of the rest of the article:}} The rest of the article is organized as follows. In Section \ref{section Space X} we first review the construction of the metric $\mathbf D$ on $\X$ from \cite{MV14}, record 
its salient properties, prove the requisite (semi)continuity properties of 
functionals on $\X$ and derive a suitable representation of the free energy. In Section \ref{section time dependence}, we derive the continuity properties of the transition probabilities in $\X$ and obtain 
a variational formula for the free energy. In Section \ref{sec-final-details}, we provide the necessary details to conclude the proof of Theorem \ref{APA} and in Appendix \ref{sec-appendix} we recall and sketch the proof of some auxiliary results.

\section{Functionals on the metric space $(\X,\mathbf{D})$ and their properties.}
\label{section Space X}

\subsection{The space $\X$ and its metric $\mathbf D$.}\label{sec-review-MV14}

Throughout the article, we will denote by $\Mcal_1= {\Mcal_1}(\R^d)$ (resp., $\Mcal_{\leq 1}$) the space of probability (resp., subprobability) distributions on $\R^d$ and by $\widetilde\Mcal_1= \Mcal_1 \big/ \sim$ the quotient space 
of $\Mcal_1$ under the action of $\R^d$ (as an additive group on $\Mcal_1$), that is, for any $\mu\in \Mcal_1$, its {\it{orbit}} is defined by $\widetilde{\mu}=\{\mu\star\delta_x\colon\, x\in \R^d\}\in \widetilde\Mcal_1$. 

As usual, we write $\alpha_n\Rightarrow\alpha$ when $\alpha_n$  converges weakly to $\alpha $ in the space $\Mcal_{\leq 1}$ (i.e., if $\int f \d\alpha_n \to \int f \d\alpha$ for all continuous and bounded $f$ in $\R^d$). 
We say two sequences $(\alpha_n)_n$ and $(\beta_n)_n$ in $\Mcal_{\leq 1}$ are {\it{widely separated}} if $\int_{\R^{2d}} F(x-y) \alpha_n(\d x) \beta_n(\d y)\to 0$ for any continuous function $F$ which vanishes at infinity.
We also say that a sequence $(\beta_n)_n$ in $\Mcal_{\leq 1}$ {\it{total disintegrates}} if for any $r>0$, $\sup_{x\in\R^d} \beta_n(B_r(x)) \to 0$ as $n\to\infty$. Clearly, any totally disintegrating sequence is widely separated 
from every sequence of subprobability measures. 

We define 
\begin{equation}\label{eq-space-X}
\X=\left\{\xi:\xi=\{\widetilde{\alpha}_i\}_{i\in I},\alpha_i\in \mathcal{M}_{\leq 1},\sum_{i\in I}\alpha_i(\R^d)\leq 1\right\}
\end{equation}
to be the space of all empty, finite or countable collections of orbits of subprobability measures with total masses $\leq 1$.
For any $\xi=(\widetilde\alpha_i)_i\in \X$ and any $\mu\in \Mcal_1(\R^d)$, we will also write 
\begin{equation}\label{xi-star-mu}
\xi\star\mu=\big(\widetilde{\alpha_i\star \mu}\big)_i.
\end{equation}
The space $\X$ also comes with a metric structure that allows explicit computations which will be used throughout the sequel. The definition of the metric is inspired by the following 
class of functionals. For any $k\geq 2$, let $\mathcal H_k$ is the space of functions $h:\left(\R^d\right)^k\rightarrow \R$ which are invariant under rigid translations and which vanish at infinity in the following sense.
Any $h \in \mathcal H_k$ satisfies 
$$
\begin{aligned}
&h(x_1+y,...,x_k+y)=h(x_1,...,x_k)\qquad \forall \,\, y,x_1,...,x_k\in\R^d
\quad \mbox{and} \\
&\lim_{\sup_{i\neq j }|x_i-x_j|\rightarrow\infty}h(x_1,...,x_k)=0.
\end{aligned}
$$
Then for $k\geq 2$, $(\mathcal H_k,\|\cdot\|_\infty)$ is a separable Banach space. Moreover, for any $h\in \mathcal{H}=\bigcup_{k\geq 2}\mathcal{H}_k$, the functionals
\begin{equation}\label{Lambda-def}
\Lambda (h,\xi)=\sum_{\widetilde\alpha\in\xi}\int_{(\R^d)^k }h(x_1,\ldots, x_k)\alpha(\d x_1)\cdots\alpha(\d x_k),
\end{equation}
are well-defined on $\X$ because of translation-invariance of $h$, and are natural continuous functions to consider on $\X$. In other words, 
a sequence $\xi_n$ is desired to ``converge" to $\xi$ in the space $\X$ provided $\Lambda(h,\xi_n)\to \Lambda(h,\xi)$ for any $h\in \mathcal H$. This leads to the following definition of the metric $\mathbf D$ on $\X$.
For any $\xi_1,\xi_2\in \X$, we set 
\begin{equation}\label{eq-D}
\begin{aligned}
\mathbf{D}(\xi_1,\xi_2)&=\sum_{r=1}^{\infty}\frac{1}{2^r}\frac{1}{1+\lVert h_r\rVert_{\infty}} \bigg|\Lambda(h_r,\xi_1)- \Lambda(h_r,\xi_2)\bigg| \\
&=\sum_{r=1}^{\infty}\frac{1}{2^r}\frac{1}{1+\lVert h_r\rVert_{\infty}}\bigg|\sum_{\widetilde{\alpha}\in \xi_1}\int h_r(x_1,...,x_{k_r})\prod_{i=1}^{k_r}\alpha(\mathrm{d}x_i)-\sum_{\widetilde{\alpha}\in \xi_2}\int h_r(x_1,...,x_{k_r})\prod_{i=1}^{k_r}\alpha(\mathrm{d}x_i)\bigg|. 
\end{aligned}
\end{equation}
\renewcommand{\thefootnote}{$\diamond$}
The following representation theorem was derived in \cite[Theorem 3.1, Theorem 3.2]{MV14}\footnote{Note that the uniqueness of the above representation theorem 
is captured by the fact that $\xi_1=\xi_2$ in $\X$ if $\Lambda(h,\xi_1)=\Lambda(h,\xi_2)$ for all $h\in \mathcal H_k$ and $k\geq 2$, while the existence part is underlined 
by the fact that for any $(\widetilde\mu_n)\subset\widetilde\Mcal_1(\R^d)$, $\Lambda(h,\mu_n)\to \Lambda(h,\xi)$ for some $\xi\in \X$ and conversely for any $\xi\in \X$ the latter convergence 
holds for some $(\widetilde\mu_n)\subset\widetilde\Mcal_1(\R^d)$.} and will be used throughout the sequel. 
\begin{theorem}\label{thm-MV14}
$\mathbf D$ is a metric on $\X$. The quotient space $\widetilde\Mcal_1$ is dense in $(\X, \mathbf D)$ and any sequence in $\widetilde\Mcal_1$ has a limit in $\X$ along a subsequence. 
Thus, $\X$ is the compactification as well as the completion of the totally bounded metric space $\widetilde\Mcal_1$ under $\mathbf D$.
\end{theorem}


The metric $\mathbf D$ provides the following convergence criterion in $\X$. 
Let a sequence $(\xi_n)_n$ consist of a single orbit $\widetilde\gamma_n$ and $\mathbf D(\xi_n,\xi)\to 0$ where $\xi=(\widetilde\alpha_i)_i\in \X$ such that $\alpha_1(\R^d)\geq \alpha_2(\R^d) \geq \dots$.
Then given any $\eps>0$, we can find $k\in \N$ such that $\sum_{i>k} \alpha_i(\R^d) <\eps$ and 
\begin{itemize}
\item We can write 
\begin{equation}\label{decompose-1}
\gamma_n= \sum_{i=1}^k\alpha_{n,i}+ \beta_n
\end{equation} 
such that for any $i=1,\dots,k$, there is a  sequence $(a_{n,i})_n\subset \R^d$ such that
\begin{equation}\label{decompose-2}
\begin{aligned}
&\alpha_{n,i}\star \delta_{a_{n,i}}  \Rightarrow \alpha_i \quad\mbox{with}\quad \lim_{n\to\infty} \, \inf_{i\ne j}\,\, |a_{n,i}- a_{n,j}| =\infty.\\
&
\end{aligned}
\end{equation}
\item The sequence $\beta_n$ {\it{totally disintegrates}}, meaning for any $r>0$, $\sup_{x\in \R^d} \beta_n\big(B_r(x)\big)\to 0$. Equivalently, for any $h\in \mathcal H_2$, 
\begin{equation}\label{decompose-3}
\begin{aligned}
&\lim_{n\to\infty} \int_{\R^{2d}} h(x,y) \alpha_{n,i}(\d x) \beta_n(\d y)=0\qquad\mbox{for any}\,\,i=1,\dots,k\qquad\mbox{and,}\\
& \limsup_{n\to\infty} \int_{\R^{2d}} h(x,y) \beta_{n}(\d x) \beta_n(\d y)\leq \eps.
\end{aligned}
\end{equation}
\end{itemize}

Finally, we remark on the topology on the space of probability measures on the space $\X$, which, as usual, will be denoted by $\Mcal_1(\X)$. 
On this space, we will work with the Wasserstein metric defined by
\begin{align}
	\label{Wasserstein-metric}
	\mathscr{W}(\vartheta,\vartheta^\prime)=\inf_{\Gamma}\int_{\X\times\X}\mathbf{D}(\xi_1,\xi_2)\gamma(\mathrm{d}\xi_1,\mathrm{d}\xi_2)
\end{align}
where the infimum is taken over  probability measures $\Gamma$ on $\X\otimes\X$ with marginals $\vartheta,\vartheta^\prime\in\Mcal_1(\X)$. Sometimes it will be convenient 
to use the dual-representation 
\begin{equation}\label{Wasserstein-dual}
\mathscr{W}(\vartheta,\vartheta^\prime)=\sup_\ell \bigg| \int_{\X} \ell(\xi) \vartheta(\xi) - \int_{\X} \ell(\xi) \vartheta^\prime(\d\xi)\bigg|  
\end{equation}
with the supremum being taken over all Lipschitz functions $\ell: \X\to \R$ with Lipschitz constant bounded by $1$. Since the difference of the integrals above is not altered by adding any finite constant, we can as well restrict the above supremum to those $\ell$ which also vanish at $\widetilde 0\in \X$. The space of such Lipschitz functions on $\X$ will be denoted by $\mathrm{Lip}_1(0)$.

\subsection{The total mass functional.}\label{subsec-total-mass}

We introduce the following functionals $\Psi, \Psi_\eps:\X\rightarrow [0,1]$ and $\mathscr I_{\Psi_\eps}:\Mcal_1(\X)\to \R$:
as
\begin{align}
&\Psi(\xi)=\sum_{i\in I}\alpha_i(\R^d) \qquad \mbox{with}\quad \xi=(\widetilde\alpha_i)_{i\in I}, \label{Psi-def}\\
&\Psi_{\eps}(\xi)=\sum_{i\in I}\int_{\R^d}\, \1_{\big\{y\in\R^d\colon\,  \alpha_i(B_1(y))>c_0\eps\big\}}(x)\,\alpha_i(\mathrm{d}x) \quad\mbox{where}\quad \eps\in (0,1), \, c_0=|B_1(0)|, \nonumber \\
&\mathscr I_{\Psi_{\eps}}(\vartheta)=\int_{\X}\Psi_{\eps}(\xi)\vartheta(\mathrm{d}\xi). \label{scr-I-psi}
\end{align}
Obviously, for any $z\in \R^d$ and $i\in I$, $\alpha_i(\R^d)=\alpha_i\star\delta_z(\R^d)$ and
$$
\begin{aligned}
\int_{\R^d} \1_{\big\{y: {(\alpha_i\star\delta_z)(B_1(y))}>c_0 \eps\big\}}(x)(\alpha_i\star\delta_z)(\mathrm{d}x)
		 	&=\int_{\R^d}\1_{\left\{y:{\alpha_i(B_1(y))}>c_0 \eps\right\}}(x)\alpha_i(\mathrm{d}x).
			\end{aligned}
			$$
Thus $\Psi$, $\Psi_\eps$ and $\mathscr I_{\Psi_\eps}$ are all well-defined.  We make two remarks regarding the above functionals. First, if 
	$$
	\nu_T=\frac 1 T \int_0^T \delta_{\widetilde{\bQ}_{\beta,t}} \, \d t \in \Mcal_1(\X) \qquad\mbox{where}\,\, \,\widetilde{\bQ}_{\beta,t}= \widetilde{\hP_{\beta,t}  W_t^{-1}} \in \X,
	$$ 
	then the identity 
	
	\begin{equation}\label{APA-psi-eps}
	\mathscr I_{\Psi_{\eps}}(\nu_T)= \frac 1 T \int_0^T \Psi_\eps(\widetilde\bQ_{\beta,t}) \d t= \frac 1 T  \int_0^T \bQ_{\beta,t}[U_{t,\eps}] \, \d t 
\end{equation}
will be useful in deriving Theorem \ref{APA}. Second,  for any $p\in \N$ and $t,\beta>0$, 
\begin{equation}\label{rmk-moment-total-mass}
\E[(\Psi(\xi)Z_{\beta,t}+(1-\Psi(\xi))\E Z_{\beta,t})^{-p}] <\infty.
\end{equation}
Indeed, if $[0,1]\ni\Psi(\xi)\geq 1/2$, 
		then $\Psi(\xi)Z_{\beta,t}+(1-\Psi(\xi))\E Z_{\beta,t}\geq \frac{1}{2}Z_{\beta,t}$
		and  by Jensen's inequality, $\E[Z_{\beta,t}^{-p}]\leq \E\big[\bE_0[\exp\{-p\beta\mathscr H_t(W)\}]\big]$, such that 
		$$
		\E\big[\big(\Psi(\xi)Z_{\beta,t}+(1-\Psi(\xi))\E Z_{\beta,t}\big)^{-p}\big]\leq 2^{p}\E\big[Z_{\beta,t}^{-p}\big] \leq 2^{p}\e^{p^2\beta^2 t V(0)/2}. 
		$$
		If $\Psi(\xi)\leq 1/2$, then also $\Psi(\xi)Z_{\beta,t}+(1-\Psi(\xi))\E Z_{\beta,t}\geq \frac{1}{2}\E Z_{\beta,t}$ which again ensures the validity of the above bound and proves \eqref{rmk-moment-total-mass}.

\renewcommand{\thefootnote}{$\diamond\diamond$}
Although $\Psi$ and $\Psi_\eps$ need not be continuous\footnote{Obviously, if $\mu_n=N(0,n)$, then $\widetilde\mu_n \to \widetilde 0$ in $\X$, while $1=\Psi(\widetilde\mu_n)>\Psi(\widetilde 0)=0$.}, we have the following.
\begin{lemma}
	\label{lower semi-cont}
	Fix $\eps\in (0,1)$. Then $\Psi$, $\Psi_{\eps}$ are lower semicontinuous on $\X$ and $\mathscr I_{\Psi_\eps}$ is lower semicontinuous on $\Mcal_1(\X)$. 
	
\end{lemma} 
\begin{proof}
	
		 If $\xi_n\to \xi=(\widetilde\alpha_i)_i$ in $\X$, we will show that $\liminf_{n\to\infty} \Psi(\xi_n) \geq \Psi(\xi)$. Suppose $\xi_n$ consists of a single orbit $\widetilde\gamma_n$. Then 
		  by the convergence criterion in $\X$ (recall \eqref{decompose-1}-\eqref{decompose-3}), for any arbitrary $\eta>0$ there exists $k\in \N$ such that, $\sum_{i>k} \alpha_i(\R^d)<\eta$ and 
		  \begin{align*}
			\gamma_n(\R^d)
			&=\sum_{i=1}^k\int_{\R^d}(\alpha_{n,i}\star \delta_{a_{n,i}})(\mathrm{d}x)+\int_{\R^d}\beta_n(\mathrm{d}x)\\
			&\geq \sum_{i=1}^k\int_{\R^d}(\alpha_{n,i}\star \delta_{a_{n,i}})(\mathrm{d}x),
		\end{align*}		
		
		and $\alpha_{n,i}\star \delta_{a_{n,i}} \Rightarrow \alpha_i$ for $i=1,\dots, k$. Therefore, 
		$$\liminf_{n\rightarrow\infty}\gamma_n(\R^d)\geq\sum_{i=1}^k\int_{\R^d}\alpha_i(\mathrm{d}x)\geq\sum_{i\in I}\alpha_i(\R^d)-\eta,$$
		and since $\eta>0$ is arbitrary, 
		$\liminf_{n\rightarrow\infty}\gamma_n(\R^d)\geq \sum_{i\in I}\alpha_i(\R^d)=\Psi(\xi)$.		
		Now if $\xi_n$ consist of multiple orbits $(\widetilde{\gamma}_{n,i})_{i\in I}$, we can choose a subsequence such that for each $i$, $\widetilde{\gamma}_{n,i}$ has a limit $(\widetilde{\alpha}_{j,i})_{j\in J} \in \X$,
		and from the first case we have
		$\sum_{j}\alpha_{j,i}(\R^d)\leq \liminf_{n\rightarrow\infty}\gamma_{n,i}(\R^d)$ for any $i$. Then by
		 Fatou's lemma 
		 		$$\Psi(\xi)=\sum_{i}\sum_{j}\alpha_{j,i}(\R^d)\leq\sum_{i }\liminf_{n\rightarrow\infty}\gamma_{n,i}(\R^d)\leq\liminf_{n\rightarrow\infty}\sum_{i }\gamma_{n,i}(\R^d)=\liminf_{n\rightarrow\infty}\Psi(\xi_n)$$
		proving lower semicontinuity of $\Psi$.
		
		For $\Psi_{\eps}$, we proceed in a similar way. 
		Since
		\begin{align*}
			\int_{\R^d}\1_{\{y\in\R^d:\gamma_n(B_1(y))>c_0\eps\}}(x)\alpha_{n,i}(\mathrm{d}x)&\geq\int_{\R^d}\1_{\{y\in\R^d:\alpha_{n,i}(B_1(y))>c_0 \eps \}}(x)\alpha_{n,i}(\mathrm{d}x)\\
			&=\int_{\R^d}\1_{\{y\in\R^d:(\alpha_{n,i}\star\delta_{a_{n,i}})(B_1(y))>c_0 \eps\}}(x)(\alpha_{n,i}\star \delta_{a_{n,i}})(\mathrm{d}x),
		\end{align*}
		for any $\eta>0$, 
		$$\liminf_{n\rightarrow\infty}\int_{\R^d}\1_{\{y\in\R^d:\gamma_n(B_1(y))>c_0\eps\}}(x)\gamma_n(\mathrm{d}x)\geq\sum_{i\in I}\int_{\R^d}\1_{\{y\in\R^d:\alpha_i(B_1(y))>c_0\eps\}}(x)\alpha_i(\mathrm{d}x)-\eta.$$
		Repeating the argument for $\Psi$ yields lower semicontinuity of $\Psi_\eps$ on $\X$, which in turn implies the lower semicontinuity of $\mathscr I_{\Psi_\eps}$ on $\Mcal_1(\X)$. 
	
\end{proof}

\subsection{The functional $\Phi$.}\label{subsec-Phi}
Recall that $V=\phi\star \phi$. We define $\Phi:\X\rightarrow \R$ by
\begin{align}
\label{function Phi}
\Phi(\xi)=\frac{\beta^2}{2}V(0)\left(1-\frac{1}{V(0)}\sum_{i\in I}\int_{\R^d\times\R^d}V(x_2-x_1)\prod_{j=1}^2\alpha_i(\mathrm{d}x_j)\right)
\end{align}
for $\xi=(\widetilde{\alpha}_i)_{i\in I}$. Again, because of shift-invariance of the integrand in the above display, $\Phi$ is well-defined on $\X$. Also, 
since $\phi$ is rotationally symmetric, for any $\alpha\in \Mcal_{\leq 1}(\R^d)$, by Cauchy-Schwarz inequality, 
\begin{equation}\label{function Phi-2}
\begin{aligned}
\int_{\R^{2d}} V(x_1-x_2) \, \alpha(\d x_1)\, \alpha(\d x_2) &=\int_{\R^{2d}} \alpha(\d x_1)\alpha(\d x_2) \, \int_{\R^d} \d z \, \phi(x_1-z) \, \phi(x_2-z) \\
&\leq \int_{\R^{2d}} \alpha(\d x_1)\alpha(\d x_2) \, \bigg[\int_{\R^d} \d z \phi^2(x_1-z)\bigg]^{1/2}  \,\, \bigg[\int_{\R^d} \d z \phi^2(x_2-z)\bigg]^{1/2} \\ &\leq \|\phi\|_2^2 = V(0)
\end{aligned}
\end{equation}
and as $\sum_{i\in I}\int_{\R^d\times\R^d}V(x_2-x_1)\prod_{j=1}^2\alpha_i(\mathrm{d}x_j)\leq  V(0)$ by the same argumentation, $\Phi(\cdot)\geq 0$. We will now state the following.
\begin{lemma}
	\label{lower semi cont for Phi}
	$\Phi$  is continuous on $\X$. 
\end{lemma}
{

	The proof of Lemma \ref{lower semi cont for Phi} follows by showing lower and upper semicontinuity.
The arguments are similar to the proof of Lemma \ref{lower semi-cont} and are omitted to avoid repetition.

}

\subsection{The partition function and the free energy.}\label{sec-free-energy}

For notational brevity, henceforth we will fix the disorder parameter $\beta>0$, and for any $t>0$, we will write 
$$
\begin{aligned}
&\hP_t=\hP_{\beta,t}, \qquad \bQ_t=\hP_{\beta,t} \, W_t^{-1}\\
& Z_{t}[x]=Z_{\beta,t}[x]=\bE_x\big[\exp\{\beta\mathscr H_t(W,B)\}\big],\qquad Z_t=Z_t[0]. 
\end{aligned}
$$ 
Likewise, $\widetilde\bQ_t\in \X$ will stand for the GMC endpoint orbit of $\bQ_t$ embedded in $\X$.

In this section, we will provide a decomposition of the ``free energy" $\frac 1T\log Z_T$ in terms of a martingale and an additive functional of $\widetilde\bQ_t$. Recall the map $\Phi$ from \eqref{function Phi}. 

\begin{lemma}[Rewrite of the free energy]
	\label{Rewrite of the partition function}
	We can write
	\begin{equation}\label{rewrite-partition}
	\begin{aligned}
	&\frac 1T \log Z_T=\frac 1 TM_T+\frac 1 T  \int_0^T\Phi(\widetilde{\bQ}_t) \d t 
	\end{aligned}
	\end{equation}
	where
	
	$$
	M_T=\beta\int_0^T\int_{\R^d}E^{\hP_t}\big[\phi(y-W_t)\big]\dot{B}(t,y) \d y  \,\d t\quad\mbox{and}\\
$$
is a square integrable martingale. In particular, 
	\begin{equation}\label{M_T}
        \begin{aligned}
	&\frac 1 T \E[\log Z_T]=\frac 1 T \int_0^T \E\big[\Phi(\widetilde\bQ_t)\big] \d t \quad\mbox{and}\\
	 &\frac 1 T \log Z_T- \frac 1 T \int_0^T \Phi(\widetilde\bQ_t) \d t \to 0\quad\mbox{a.s.}\,\,-\P.
	\end{aligned}
	\end{equation}
	\end{lemma}
\begin{proof}
	Recall our earlier notation
	\begin{equation}\label{scr-H}
	\mathscr H_t(W,B)=\int_0^t\int_{\R^d}\phi(y-W_s)\dot{B}(s,y)\d y \,\, \mathrm{d}s.
        \end{equation}
	 We now apply It\^{o}'s formula to $Z_t=\bE_0\big[\e^{\beta \mathscr H_t(W,B)}\big]$
	to get
	\begin{align*}
	\mathrm{d}Z_t=\bE_0\bigg[\beta\int_{\R^d}\e^{\beta \mathscr H_t(W,B)}\phi(y-W_t)\dot{B}(t,y) \d y \bigg] \d t +\bE_0\bigg[\frac{\beta^2}{2}\int_{\R^d}\e^{\beta \mathscr H_t(W,B)}\phi(y-W_t)^2\mathrm{d}y\bigg]\mathrm{d}t.
	\end{align*}
	We can also compute the quadratic variation for $Z_t$ as 
	\begin{align*}
	\mathrm{d}\langle Z_t\rangle&=\mathrm{d}\bigg\langle  \bE_0\bigg[\beta \int_{\R^d}\, \e^{\beta \mathscr H_t(W,B)}\phi(y-W_t)\dot{B}(t,y) \d y \bigg]\bigg{\rangle}
	&=\beta^2\bE_0^{\otimes }\left[\int_{\R^d}\, \e^{\beta(\mathscr H_t(W,B)+\mathscr H_t(W^{\prime},B))}\phi(y-W_t)\phi(y-W_t^{\prime})\right] \d t\\
	&=\beta^2\bE_0^{\otimes }\left[V(W_t-W_t^{\prime})\, \e^{\beta (\mathscr H_t(W,B)+\mathscr H_t(W^{\prime},B))}\right] \d t
	\end{align*}
	where $W^{\prime}$ is another Brownian motion independent of $W$. We again apply It\^{o}'s formula to $\log Z_t$ and use the last display to get 
	\begin{align*}
	\mathrm{d}&\log Z_t=\frac{1}{Z_t}\mathrm{d}Z_t-\frac{1}{2Z_t^2}\mathrm{d}\langle Z_t\rangle\\
	&=\beta E^{\hP_t}\left[\int_{\R^d}\phi(y-W_t)\dot{B}(t,y) \d y\right]\mathrm{d}t+\frac{\beta^2}{2}E^{\hP_t}\left[\int_{\R^d}\phi(y-W_t)^2\mathrm{d}y\right]\mathrm{d}t-\frac{\beta^2}{2}E^{\hP_t^{\otimes }}\left[V(W_t-W_t^{\prime})\right]\mathrm{d}t.
	\end{align*}
	Since $\int_{\R^d} \phi(y-W_t)^2 \d y=\int_{\R^d} \phi^2(y)\d y= V(0)$, then 
	\begin{align*}
	{\log Z_T}&= \int_0^T\beta E^{\hP_t}\left[\int_{\R^d}\phi(y-W_t)\dot{B}(t,y)\d y\right]\mathrm{d}t+\frac{\beta^2 T V(0)}2
	-\int_0^T\frac{\beta^2}{2}E^{\hP_t^{\otimes }}\left[V(W_t-W_t^{\prime})\right]\mathrm{d}t\\
	&=\beta \int_0^T \int_{\R^d}E^{\hP_t}\left[\phi(y-W_t)\right]\dot{B}(t,y) \, \d y \, \d t\\
	&\qquad +\int_0^T \d t \bigg[\frac{\beta^2}{2}V(0)\left(1-\frac{1}{V(0)}\int_{\R^d\times\R^d}V(x_2-x_1)\hP_t(W_t\in\mathrm{d}x_1)\hP_t(W_t^{\prime}\in\mathrm{d}x_2)\right)\bigg] \\
	&=M_T+ \int_0^T \Phi(\widetilde\bQ_t)\d t
	\end{align*}
	proving \eqref{rewrite-partition}. 
	
	Now the first display in \eqref{M_T} readily follows since $M_T$ is a martingale, whose quadratic variation is given by
	\begin{align*}
	{\mathrm{d}\langle M_T\rangle}=\int_0^T\int_{\R^d}\beta^2\left(E^{\hP_t}\left[\phi(y-W_t)\right]\right)^2\mathrm{d}y\mathrm{d}t
	&\leq\int_0^T\int_{\R^d}\beta^2E^{\hP_t}\left[\phi(y-W_t)^2\right]\mathrm{d}y\mathrm{d}t\\
	&= \beta^2 \int_0^T \d t E^{\hP_t}\bigg[\int_{\R^d} \d y\phi(y-W_t)^2\bigg] \\
	&=T\beta^2 V(0).
	\end{align*}
Since $M_T/T\to 0$ almost surely,  the second display in \eqref{M_T} follows from \eqref{rewrite-partition}. 
\end{proof}

We will end this section with a corollary which will be used later. For the map $\Phi : \X \to \R$, we define the functional, $\mathscr I_\Phi: \Mcal_1(\X) \to \R$ as 
\begin{equation}\label{scr-I}
\mathscr I_\Phi(\vartheta)= \int_{\X} \Phi(\xi) \, \vartheta(\d\xi).
\end{equation}
Again since $\Phi$ is continuous on the compact metric space $\X$, $ \mathscr I_\Phi(\cdot)$ is continuous on $\Mcal_1(\X)$. 
\begin{cor}\label{cor-I-Phi}
With 
$$
\nu_T= \frac 1 T \int_0^T \delta_{\widetilde\bQ_t} \,\, \d t \in \Mcal_1(\X) 
$$
we have
$$
\begin{aligned}
&\frac 1 T \mathbf E[ \log Z_T]= \mathbf E[\mathscr I_\Phi(\nu_T)], \\
&\liminf_{T\to\infty} \frac 1 T \log Z_T=\liminf_{T\to\infty} \mathscr I_\Phi(\nu_T)  \quad \mathrm{a.s.}
\end{aligned}
$$
\end{cor}
\begin{proof}
Both statements follow immediately from Lemma \ref{Rewrite of the partition function} and the definition of $\nu_T$.
\end{proof}

\section{Dynamics on Elements of $\X$.}
\label{section time dependence}

\subsection{Continuity of the transition probabilities.}

Recall the notation for $\mathscr H_t(W,B)$ from \eqref{scr-H}. We fix $t>0$, and for any element $\xi=(\widetilde\alpha_i)_i \in \X$, we set 
\begin{equation}\label{alpha after time t}
        \alpha_i^{\ssup t}(\mathrm{d}x):=\frac{1}{\mathscr F_t(\xi)+\E\big[ Z_t-\mathscr F_t(\xi)\big]}\int_{\R^d}\alpha_i (\mathrm{d}z)\bE_z\bigg [\1_{\{W_t\in\mathrm{d}x\}}\, \exp\big\{\beta\mathscr H_t(W,B)\big\}\bigg] 
	\end{equation}
	where 
 
	\begin{equation}\label{scr-F}
	\begin{aligned}
	&\mathscr F_t(\alpha_i)=\int_{\R^d}\int_{\R^d}\alpha_i(\mathrm{d}z) \bE_z\bigg[\1_{\{W_t\in\mathrm{d}x\}}\,\, \exp\big\{\beta\mathscr H_t(W,B)\big\}\bigg] \qquad\mbox{and} \\
	&\mathscr F_t(\xi)=\sum_i \mathscr F_t(\alpha_i).
	\end{aligned}
        \end{equation}
        We remark that  for any $a\in \R^d$ and $t>0$, $\mathscr F_t(\alpha_i)\overset{\ssup d}=\mathscr F_t(\alpha_i\star\delta_a)$ and 
$$
(\alpha_i\star\delta_a)^{\ssup t}(\mathrm{d}x)\overset{(d)}{=}(\alpha_i^{\ssup t}\star\delta_a)(\mathrm{d}x),
$$
and for any $r,t>0$, $(\widetilde\delta_0)^{\ssup t}\overset{\ssup d}=\widetilde\bQ_t$ and 
$
\widetilde{\bQ}_t^{\ssup r}\overset{\ssup d}=\widetilde{\bQ}_{t+r}
$
 since
	\begin{align*}
	&\bQ_{t+r}(\mathrm{d}x)\\&= \frac{1}{Z_{t+r}}\int_{\R^d}\bE_0\bigg{[}\e^{\beta\int_0^t\phi(y-W_s)\dot{B}(s,\mathrm{d}y)\mathrm{d}s}\1_{\{W_t\in\mathrm{d}z\}}\bE_0\bigg(\e^{\beta\int_t^{t+r}\phi(y-W_s)\dot{B}(s,\mathrm{d}y)\mathrm{d}s}\1_{\{W_{t+r}-W_t\in \mathrm{d}(x-z)\}}\bigg{|}\mathcal{G}_t\bigg)\bigg{]} \\
	&\overset{\ssup d}=\int_{\R^d}\frac 1 {\mathscr F_r^\prime(\widetilde{\bQ}_t)} \bE_z\bigg[\e^{\beta\int_0^r\phi(y-W_s^{\prime})\dot{B}^{}(s,\mathrm{d}y)\mathrm{d}s}\1_{\{W_r^{\prime}\in \mathrm{d}x\}}\bigg]\bQ_{t}(\mathrm{d}z),
        \end{align*}
        where $\mathcal G_t$ is the $\sigma$-algebra generated by the Brownian path $W$ until time $t$ and $\mathscr F_r^\prime$ is defined as $\mathscr F_r$, but w.r.t. a Brownian path $W^\prime$ independent of $W$.

     Then \eqref{alpha after time t} and the above remarks, for any $t>0$ and $\xi=(\widetilde\alpha_i)_i\in \X$, define a transition kernel $\pi_t(\xi, \cdot)\in \Mcal_1(\X)$ as
\begin{align}
	\label{pi-function}
	\pi_t(\xi,\mathrm{d}\xi^{\prime})=\P\big[\xi^{\ssup t}\in\mathrm{d}\xi^\prime |\xi\big] \qquad \mbox{where}\quad \xi^{\ssup t}= \big(\widetilde\alpha_i^{\ssup t}\big)_{i\in I} \in \X.
	\end{align}

Here is the main result of this section. 

\begin{theorem}
	\label{continuity of the pi-function}
	For any fixed $t>0$, the map
	\begin{align*}
		\pi_t:\X&\rightarrow\mathcal{M}_1(\X)\\
		\xi&\mapsto\pi_t(\xi,\cdot)
	\end{align*}
	is continuous with respect to the Wasserstein metric on $\Mcal_1(\X)$. 
\end{theorem}

The rest of the section is devoted to the proof of Theorem \ref{continuity of the pi-function}. 

\noindent{\bf{Proof of Theorem \ref{continuity of the pi-function}.}} We want to show that if $\xi_n \to \xi$ in $(\X,\mathbf D)$, then for any fixed $t>0$, $\pi_t \xi_n \to \pi_t \xi$ in $(\Mcal_1(\X), \mathscr W)$ where $\mathscr W$ is the Wasserstein metric  $\mathscr W$ (recall \eqref{Wasserstein-metric}). 
Since $\mathscr W(\pi_t\xi_n, \pi_t \xi) \leq \mathbf E\big[\mathbf D(\pi_t \xi_n,\pi_t\xi)\big]$, by definition of the metric $\mathbf D$, it suffices to show that for any $h_r\in \mathcal H_{k_r}$ and $r\geq 1$, 
\begin{align}
	\begin{split}
		\label{claim1-theorem}
		\E\bigg[ \bigg| \Lambda(h_r,\xi_n^{\ssup t})- \Lambda(h_r,\xi^{\ssup t})\bigg| \bigg]=
		\E\bigg{|}&\sum_{\widetilde{\alpha}\in\xi_n^{\ssup t}}\int_{(\R^d)^{k_r}}h_r(x_1,...,x_{k_r})\prod_{i=1}^{k_r}\alpha(\mathrm{d}x_i)\\
		&-\sum_{\widetilde{\alpha}\in\xi^{\ssup t}}\int_{(\R^d)^{k_r}}h_r(x_1,...,x_{k_r})\prod_{i=1}^{k_r}\alpha(\mathrm{d}x_i)\bigg{|}\\
		&\longrightarrow 0\qquad\mbox{as}\quad n\to\infty.
		\end{split}
	\end{align}
	Note that by \eqref{alpha after time t}, the first term $\Lambda(h_r,\xi_n^{\ssup t})$ in the above display can be rewritten as 
	$$
	\begin{aligned}
	&\sum_{\widetilde{\alpha}\in\xi_n}\int_{(\R^d)^{k_r}}h_r(x_1,...,x_{k_r})\prod_{i=1}^{k_r}\alpha^{\ssup t}(\mathrm{d}x_i) \\
	&=\sum_{\widetilde{\alpha}\in\xi_n}\int_{(\R^d)^{k_r}}h_r(x_1,...,x_{k_r})\prod_{i=1}^{k_r}\bigg(\frac 1 {\mathscr F_t(\xi_n)+ \E[Z_t- \mathscr F_t(\xi_n)]} \int_{\R^d} \alpha(\d z_i) \mathbb E_{z_i}\big[\1_{\{W_t\in \d x_i\}}\, 
	\exp\big\{\beta\mathscr H_t(W,B)\big\}\big]\bigg) \\
	&= \bigg[\frac {1} {\mathscr F_t(\xi_n)+ \E[Z_t- \mathscr F_t(\xi_n)]}\bigg]^{k_r} \,\, \sum_{\widetilde{\alpha}\in\xi_n}\int_{(\R^d)^{k_r}}h_r(x_1,...,x_{k_r})\prod_{i=1}^{k_r} \bigg(\int_{\R^d} \alpha(\d z_i) \mathbb E_{z_i}\big[\1_{\{W_t\in \d x_i\}}\, 
	\exp\big\{\beta\mathscr H_t(W,B)\big\}\big]	\bigg) \\
	&=\bigg[\frac {1} {\mathscr F_t(\xi_n)+ \E[Z_t- \mathscr F_t(\xi_n)]}\bigg]^{k_r} \,\, \Lambda(h_r,\mathscr A_t(\xi_n)),
	\end{aligned}
	$$
	where in the third identity we used the notation 
$$
\mathscr A_t(\xi)=((\mathscr F_t(\xi)+ \E[Z_t- \mathscr F_t(\xi)])\alpha_{i}^{\ssup t})_{i\in I},
$$
 recall the definition of $\widetilde\alpha^{\ssup t}_i$ from \eqref{alpha after time t} and that of $\Lambda(h,\xi)$ from \eqref{Lambda-def}. Note that if $\Psi(\xi)=0$, the second term in \eqref{claim1-theorem} is zero and since $\xi_n\rightarrow\xi$, by a similar argument as in the proof of Proposition \ref{prop-step1} below, $\Lambda(h_r,\xi^{\ssup t}_n)\rightarrow 0$ as $n\rightarrow\infty$. Thus we restrict to the case where $\Psi(\xi),\Psi(\xi_n)>0$.

In view of the last computation, then the claim \eqref{claim1-theorem} follows by triangle inequality once we prove the following two facts.

	\begin{prop}\label{prop-step1}
	For any $k_r\geq 2$, 
	\begin{equation}\label{prop-step1-eq}
	\lim_{n\to\infty}\E\bigg[\bigg(\frac {1} {\mathscr F_t(\xi)+ \E[Z_t- \mathscr F_t(\xi)]}\bigg)^{k_r}	\,\, \bigg|\Lambda(h_r,\mathscr A_t(\xi_n)) - \Lambda(h_r, \mathscr A_t(\xi))\bigg|\,\, \bigg] =0.
	\end{equation}
	\end{prop}
	
	\begin{prop}\label{prop-step2}
	For any $k_r\geq 2$, 
	\begin{equation}\label{prop-step2-eq}
	\lim_{n\to\infty}\E\bigg[\Lambda(h_r,\mathscr A_t(\xi_n)) \,\, \bigg| \bigg(\frac {1} {\mathscr F_t(\xi_n)+ \E[Z_t- \mathscr F_t(\xi_n)]}\bigg)^{k_r} - \bigg(\frac {1} {\mathscr F_t(\xi)+ \E[Z_t- \mathscr F_t(\xi)]}\bigg)^{k_r} \bigg| \, \bigg]=0.
	\end{equation}
	\end{prop}
	We will first finish the following.

	\noindent{\bf{Proof of Proposition \ref{prop-step1}:}}
	
	 Since $\xi_n \to \xi$ in $(\X,\mathbf D)$ and $\E\Big[\big|\Lambda(h_r, \mathscr A_t(\xi_n))- \Lambda(h_r,\mathscr A_t(\xi))\big|^2\Big] \leq 4 \|h_r\|_\infty^2\e^{2\beta^2tV(0)}$, by the Cauchy-Schwarz inequality and dominated convergence theorem 
	it suffices to show that, for a finite constant $C$, 
	\begin{equation}\label{eq-CS-bound}
	\E\bigg[\bigg(\frac {1} {\mathscr F_t(\xi)+ \E[Z_t- \mathscr F_t(\xi)]}\bigg)^{2k_r}\bigg] \leq C.
	\end{equation}
	Indeed, note that 
	$\mathscr F_t(\xi)+\E[Z_t-\mathscr F_t(\xi)]= \mathscr F_t(\xi)+(1-\Psi(\xi))\E Z_t$. The calculation for \eqref{rmk-moment-total-mass} with an extra argument in the case where $\Psi(\xi)\geq 1/2$ proves the inequality above. In that case, we use that $\sum_{\widetilde{\alpha}\in\xi}\int\alpha(\d x)=\Psi(\xi)$ and Jensens's inequality so that
	\begin{align}
	\begin{split}
	\label{Jensen for upper bound}
		\E\big[\mathscr F_t(\xi)^{-2k_r}\big]&=\Psi(\xi)^{-2k_r}\E\bigg[\bigg(\int Z_t[x]\Big(\sum_{\widetilde{\alpha}\in\xi}\frac{\alpha(\d x)}{\Psi(\xi)}\Big)\bigg)^{-2k_r}\bigg]\\
		&\leq \Psi(\xi)^{-2k_r}\E\bigg[\int Z_t[x]^{-2k_r}\Big(\sum_{\widetilde{\alpha}\in\xi}\frac{\alpha(\d x)}{\Psi(\xi)}\Big)\bigg]	\leq \Psi(\xi)^{-2k_r}\e^{2k_r^2\beta^2tV(0)}
	\end{split}
	\end{align}
	to complete the proof of \eqref{eq-CS-bound} and Proposition \ref{prop-step1}.

	\qed

	   We will now prove the following.
	   
	   \noindent{\bf{Proof of Proposition \ref{prop-step2}:}} In order to prove \eqref{prop-step2-eq}, we recall \eqref{scr-F} and estimate 
	   \begin{equation}\label{prop-step2-eq2}
	   \Lambda(h_r,\mathscr A_t(\xi_n)) \leq \|h_r\|_\infty \, \mathscr F_t(\xi_n)^{k_r}.
	   \end{equation}
	   Moreover, note that, $\E[Z_t- \mathscr F_t(\xi_n)] \geq 0$, since 
	   $$
	   \begin{aligned}
	   \E[\mathscr F_t(\xi_n)] =\E[Z_t]\,\,\sum_i \int_{\R^d} \int_{\R^d} \alpha_{n,i}(\d z) \bP_z[W_t\in\d x ] = \E[Z_t] \sum_i \int_{\R^d} \alpha_{n,i}(\d z)= \E[Z_t] \, \Psi(\xi_n) \leq \E[Z_t].
	   \end{aligned}	   
	   $$
	   Therefore by \eqref{prop-step2-eq2}, the requisite claim \eqref{prop-step2-eq} for Proposition \ref{prop-step2} follows once we prove the estimate stated below in Proposition \ref{prop-step3}. 
	   \qed
	   
	   \begin{prop}\label{prop-step3}
	   For any $k_r\geq 2$, 
	   $$
	   \lim_{n\to\infty} \E\bigg[\bigg| 1- \bigg(\frac{\mathscr F_t(\xi_n)+ \E[Z_t- \mathscr F_t(\xi_n)]}{\mathscr F_t(\xi)+ \E[Z_t- \mathscr F_t(\xi)]}\bigg)^{k_r}\bigg| \bigg]=0.
	   $$
	   \end{prop}
	   The proof of Proposition \ref{prop-step3} is based on the following two results. For $k\in \N$ and $\xi=(\widetilde\alpha_i)_{i=1}^k$,  we will write (recall \eqref{scr-F}), 
	   \begin{equation}\label{scr-G}
	   \begin{aligned}
	    \overline{\mathscr F}_t(\xi)= \E[Z_t] + \sum_{i=1}^k \big[\mathscr F_t(\alpha_i)- \E[\mathscr F_t(\alpha_i)]\big].
	   \end{aligned}
	    \end{equation}
	   
	   \begin{lemma}\label{lemma-step4}
	   Let $k\in \N$ and $\xi_n=(\widetilde\alpha_{n,i})_{i=1}^k$, $\xi=(\widetilde\alpha_i)_{i=1}^k$ such that $\alpha_{n,i}\star \delta_{a_{n,i}} \Rightarrow \alpha_i$ for $i=1,\dots,k$ and $|a_{n,i}-a_{n.j}|\to 0$ for $i\neq j$. If $\Psi(\xi)=\sum_{i=1}^k \alpha_i(\R^d)>0$, then for any $p\in \N$ with $p\geq 2$, 
	   $$
	   \lim_{n\to\infty}\,\,\E\bigg[\Big(\overline{\mathscr F}_t(\xi_n)^p - \overline{\mathscr F}_t(\xi)^p\Big)^2\bigg]=0.
	   	   $$
	   	   \end{lemma}
		   \begin{prop}[Second moment method]\label{lemma-step5}
		   Let $(\beta_n)_n$ be a sequence in $\Mcal_{\leq 1}(\R^d)$ that totally disintegrates (recall \eqref{decompose-3}). Then 
		   $$
		   \lim_{n\to\infty} \E\bigg[\big(\mathscr F_t(\beta_n) - \E[\mathscr F_t(\beta_n)]\big)^2\bigg]=0.
		   $$
		   \end{prop}
		   We will first prove Lemma \ref{lemma-step4} and Proposition \ref{lemma-step5} and then deduce Proposition \ref{prop-step3} from these two results. 
		   
		   \noindent{\bf{Proof of Lemma \ref{lemma-step4}:}} We recall that  
		\begin{align*}
			\mathscr F_t(\alpha_i)=\int_{\R^d}\int_{\R^d}\alpha_i(\mathrm{d}z)\bE_z\left[\1_{\{W_t\in\mathrm{d}x\}}\, \e^{\beta\int_0^t\int_{\R^d}\phi(y-W_s)\dot{B}(s,\mathrm{d}y)\mathrm{d}s}\right]\\
		\end{align*}
		and note that $\E [\mathscr F_t(\alpha_i)]=\alpha_i(\R^d)\E Z_t$. Hence, we need to show that
		\begin{equation}\label{A_n-A}
		\E\big[ \big(A_n^p -A^p\big)^2\big] \to 0
		\end{equation}
		where 
		$$
		\begin{aligned}
		& A_n= \E Z_t+\sum_{i=1}^k \int_{\R^d}\alpha_{n,i}(\mathrm{d}z)\bE_z\Big[ \e^{\beta\mathscr H_t(W,B)}-\E\big[\e^{\beta\mathscr H_t(W,B)}\big]\Big] \quad \mbox{and}\\
		&  A= \E Z_t+\sum_{i=1}^k \int_{\R^d}\alpha_i(\mathrm{d}z)\bE_z\Big[ \e^{\beta\mathscr H_t(W,B)}-\E\big[\e^{\beta\mathscr H_t(W,B)}\big]\Big].
		\end{aligned}$$
		Binomial theorem then yields
		
		$$
		A_n^p - A^p = \sum_{l=0}^{p-1} \binom p l \big(\E[Z_t]\big)^l \,\bigg(\Big(\sum_{i=1}^k \int_{\R^d} \bar Z_t[x]\alpha_{n,i}(\mathrm{d}x)\Big)^{p-l}-\Big(\sum_{i=1}^k \int_{\R^d} \bar Z_t[x]\alpha_{i}(\mathrm{d}x)\Big)^{p-l}\bigg),
		$$
		where we have used the notation 
$$
\bar Z_t[x]=\bE_x\Big[ \e^{\beta\mathscr H_t(W,B)}-\E\big[\e^{\beta\mathscr H_t(W,B)}\big]\Big].
$$
The requisite claim follows once we show 
		\begin{align}
		\label{convergence of the alphas}
			\E\bigg[\bigg(\Big(\sum_{i=1}^k \int_{\R^d} \bar Z_t[x]\alpha_{n,i}(\mathrm{d}x)\Big)^{p}-\Big(\sum_{i=1}^k \int_{\R^d} \bar Z_t[x]\alpha_{i}(\mathrm{d}x)\Big)^{p}\bigg)^2\bigg]\rightarrow 0	
		\end{align}
		for all $p\in\N$. We first consider the case $p=2$. In this case, $\Big(\sum_{i=1}^k\int_{\R^d}\bar Z_t[x]\alpha_i(\d x)\Big)^2=\sum_{i=1}^k\sum_{j=1}^k\int_{\R^d}\bar Z_t[x]\alpha_i(\d x)\int_{\R^d}\bar Z_t[x]\alpha_j(\d x)
		$ and
		\begin{align}
		\label{adding an alpha summand}
		\begin{split}
			&\int_{\R^d}\bar Z_t[x]\alpha_i(\d x)\int_{\R^d}\bar Z_t[x]\alpha_j(\d x)-\int_{\R^d}\bar Z_t[x]\alpha_{n,i}(\d x)\int_{\R^d}\bar Z_t[x]\alpha_{n,j}(\d x)\\
			&=\int_{\R^d}\bar Z_t[x](\alpha_i-\alpha_{n,i})(\d x)\int_{\R^d}\bar Z_t[x]\alpha_j(\d x)+\int_{\R^d}\bar Z_t[x]\alpha_{n,i}(\d x)\int_{\R^d}\bar Z_t[x](\alpha_j-\alpha_{j,n})(\d x).	
		\end{split}
		\end{align}
		 For the first summand on the right-hand side above, we then have
		\begin{align}
		\label{using CSI repeatedly-1}
			&\!\!\!\!\!\!\E\bigg[\Big(\sum_{i,j=1}^k\int_{\R^d}\bar Z_t[x](\alpha_i-\alpha_{n,i})(\d x)\int_{\R^d}\bar Z_t[x]\alpha_j(\d x)\Big)^2\bigg]\\
		\label{using CSI repeatedly-2}
			&\!\!\!\!\!\!=\E\bigg[\sum_{i,j,l,m=1}^k\int_{\R^d}\bar Z_t[x](\alpha_i-\alpha_{n,i})(\d x)\int_{\R^d}\bar Z_t[x]\alpha_j(\d x)\int_{\R^d}\bar Z_t[x](\alpha_l-\alpha_{n,l})(\d x)\int_{\R^d}\bar Z_t[x]\alpha_m(\d x)\bigg]\\
		\label{using CSI repeatedly-3}
			&\!\!\!\!\!\!\leq \sum_{i,j,l,m=1}^k \E\bigg[\Big(\int_{\R^d}\bar Z_t[x](\alpha_i-\alpha_{n,i})(\d x)\Big)^4\bigg]^{1/4}\E\bigg[\Big(\int_{\R^d}\bar Z_t[x]\alpha_j(\d x)\Big)^4\bigg]^{1/4}\\
			\label{using CSI repeatedly-4}
			&\!\!\!\!\!\!\,\,\,\times \E\bigg[\Big(\int_{\R^d}\bar Z_t[x](\alpha_l-\alpha_{n,l})(\d x)\Big)^4\bigg]^{1/4}\E\bigg[\Big(\int_{\R^d}\bar Z_t[x]\alpha_m(\d x)\Big)^4\bigg]^{1/4},
		\end{align}
		where we used  the Cauchy-Schwarz inequality for the upper bound. Since for $i=1,\dots,k$, 
		$$
		\Big(\int_{\R^d}\bar Z_t[x](\alpha_i-\alpha_{n,i})(\d x)\Big)^4=\int_{\R^{4d}} \prod_{j=1}^4\big[\bar Z_t[x_j](\alpha_i-\alpha_{n,i})(\d x_j)\big]
		$$
		and $\E\big[\bar Z_t[x]\bar Z_t[y]\big]\leq \E\big[\bar Z_t[x]^2\big]^{1/2}\E\big[\bar Z_t[y]^2\big]^{1/2}$ as well as $\E\big[\bar Z_t[x]^{2p}\big]\leq\E\big[Z_t[x]^{2p}\big]\leq \e^{2p^2\beta^2tV(0)}$,
		$$
		\E\bigg[\Big(\int_{\R^d}\bar Z_t[x](\alpha_i-\alpha_{n,i})(\d x)\Big)^4\bigg]\leq\e^{8\beta^2tV(0)}\big(\alpha_i(\R^d)-\alpha_{n,i}(\R^d) \big)^4.
		$$
		The last inequality can also be applied to the other factors in (\ref{using CSI repeatedly-3})-(\ref{using CSI repeatedly-4}) such that
		\begin{align}
		\label{last inequality for alphas}
		\begin{split}
			\E\bigg[&\Big(\sum_{i,j=1}^k\int_{\R^d}\bar Z_t[x](\alpha_i-\alpha_{n,i})(\d x)\int_{\R^d}\bar Z_t[x]\alpha_j(\d x)\Big)^2\bigg]\\
			&\leq \e^{8\beta^2tV(0)}\Psi(\xi)^2\sum_{i,l=1}^k (\alpha_i(\R^d)-\alpha_{n,i}(\R^d))(\alpha_l(\R^d)-\alpha_{n,l}(\R^d)).
		\end{split}
		\end{align}
		Now since $\alpha_{n,i}\star \delta_{a_{n_i}} \weak \,\,\alpha_i$, \eqref{adding an alpha summand} together with \eqref{last inequality for alphas}  yields \eqref{convergence of the alphas} for $p=2$. The same argument then carries over to the case $p\in \N$ (Indeed, for general $p$, in \eqref{adding an alpha summand} we have to add $p-1$ summands instead of one and the exponent in the upper bound of \eqref{last inequality for alphas} is then given by $2p^2\beta^2tV(0)$.)

	
		\qed
		
		We will now provide the following.
		
		\noindent{\bf{Proof of Proposition \ref{lemma-step5}:}} The proof involves two main steps.

		\noindent{\bf{Step 1: Total disintegration.}} 		
		Let $\beta_n$ be a sequence in $\Mcal_{\leq 1}(\R^d)$ which totally disintegrates, meaning that, for any $r>0$, $\sup_{x\in \R^d} \beta_n(B_r(x)) \to0$ and $\int_{\R^{2d}} h(x_1-x_2) \beta_n(\d x_1) \beta_n(\d x_2)\to 0$.
                We want to show that, for any fixed $t>0$, 
		\begin{equation}\label{claim-lemma-step5}
		\E\big[\mathscr F_t(\beta_n)^2\big]- \E[\mathscr F_t(\beta_n)]^2 \to 0. 
		\end{equation}
		
		Note that
		\begin{equation}\label{eq0-step5}
		\begin{aligned}
		\E\big[\mathscr F_t(\beta_n)^2\big] &= \E\bigg[\int_{\R^{2d}} \, \int_{\R^{2d}} \beta_n(\d z_1)\, \beta_n(\d z_2) \,\, \bE^{\otimes}_{\ssup{z_1,z_2}}\bigg[\prod_{i=1}^2\big(\1\{W^{\ssup i}_t\in \d x_i\} \exp\{\beta \mathscr H_t(W^{\ssup i},B)\}\big)\bigg]\bigg] \\
		&=(\mathrm I)+ (\mathrm {II})
		\end{aligned}
		\end{equation}
		where for any $R>0$, 
		\begin{equation}\label{def-I}
		(\mathrm I)= \E\bigg[\int_{\R^{2d}} \, \int_{B_{R}(x_1)}\,\int_{B_{R}(x_2)}\,\, \beta_n(\d z_1)\, \beta_n(\d z_2) \,\, \bE^{\otimes}_{\ssup{z_1,z_2}}\bigg[\prod_{i=1}^2\big(\1\{W^{\ssup i}_t\in \d x_i\} \exp\{\beta \mathscr H_t(W^{\ssup i},B)\}\big)\bigg]\bigg] 		
		\end{equation}
		and $(\mathrm{II})$ is defined canonically, which we can estimate using Fubini's theorem as follows: 
		\begin{equation}\label{II}
		\begin{aligned}
		(\mathrm{II)} &\leq 2  \e^{2\beta^2 t V(0)} \,\, \int_{\R^{2d}}\, \int_{\R^d} \beta_n(\d z_2) \, \bP_{z_2}[W_t\in \d x_2]\,\, \int_{B_{R}(x_1)^{\mathrm c}} \beta_n(\d z_1) \, \bP_{z_1}[W_t\in \d x_1] \\
		&\leq C \e^{2\beta^2 t V(0)} \beta_n(\R^d) \,\, 	\bP_0[W_t\in B_R(0)^{\mathrm c}] \\
		&\leq C \e^{2\beta^2 t V(0)} \,\, 	\bP_0[W_t\in B_R(0)^{\mathrm c}]\\
		&=\delta(R) \to 0\qquad\quad\mbox{as}\quad R\to\infty.
		\end{aligned}
		\end{equation}
		
		Hence we focus on \eqref{def-I}, which can be decomposed further as $(\mathrm I)= (\mathrm I)_{\mathrm a}+ (\mathrm I)_{\mathrm b}$, where 
		\begin{equation}\label{Ia-0}
		\begin{aligned}
		(\mathrm I)_{\mathrm a}&=  \E\bigg[\int_{\R^{2d}} \, \int_{B_{R}(x_1)}\,\int_{B_{R}(x_2)}\,\, \beta_n(\d z_1)\, \beta_n(\d z_2)\1\{|x_1-x_2| \geq 2R\} \\
		&\qquad\qquad\qquad\times \,\, \bE^{\otimes}_{\ssup{z_1,z_2}}\bigg[\prod_{i=1}^2\big(\1\{W^{\ssup i}_t\in \d x_i\} \exp\big\{\beta\mathscr H_t(W^{\ssup i},B)\big\}\big)\bigg]\bigg]
		\end{aligned}
		\end{equation}		
		and 
		\begin{equation}\label{Ib}
		\begin{aligned}
		(\mathrm I)_{\mathrm b}&=  \E\bigg[\int_{\R^{2d}} \, \int_{B_{R}(x_1)}\,\int_{B_{R}(x_2)}\,\, \beta_n(\d z_1)\, \beta_n(\d z_2)\1\{|x_1-x_2| \leq 2R\} \\
		&\qquad\qquad\qquad\times \,\, \bE^{\otimes}_{\ssup{z_1,z_2}}\bigg[\prod_{i=1}^2\big(\1\{W^{\ssup i}_t\in \d x_i\} \exp\big\{\beta\mathscr H_t(W^{\ssup i},B)\big\}\big)\bigg]\bigg] 		\\
		&\leq \e^{2 \beta^2 t V(0)} \,\, \int_{\R^{2d}} \int_{B_R(0)} \int_{B_R(0)}\beta_n(\d(x_1-z_1)) \beta_n(\d (x_2-z_2))\,\, \1\{|x_1-x_2| \leq 2R\}  \prod_{i=1}^2\bP_0[W_t\in \d z_i]   \\
		&\leq 
		\e^{2 \beta^2 t V(0)} \int_{\R^{2d}} \prod_{i=1}^2\bP_0[W_t\in \d z_i] \,\, \int_{\R^{2d}} 	\1\{|x_1-x_2| \leq 4R\}\,\beta_n(\d x_1)\beta_n (\d x_2) \\
		&=\e^{2 \beta^2 t V(0)} \, \int_{\R^{2d}} 	\1\{|x_1-x_2| \leq 4R\}\,\beta_n(\d x_1)\beta_n (\d x_2) 
\leq\e^{2 \beta^2 t V(0)} \, \int_{\R^{2d}} 	h_R(x_1-x_2)\beta_n(\d x_1)\beta_n (\d x_2),
		\end{aligned}
		\end{equation}
	      where $h_R(\cdot)$ is a continuous function that is identically one inside the ball of radius $4R$ around the origin and vanishes outside a ball of radius $4R+1$. Since $\beta_n$ totally disintegrates, for any fixed $t, R$, the last display converges to zero as $n\to\infty$. 
	       
	       \medskip\noindent{\bf{Step 2: Decoupling.}} We now focus on $(\mathrm I)_{\mathrm a}$ defined in \eqref{Ia-0}, which can be estimated further as follows:
	       	\begin{equation}\label{Ia}
		\begin{aligned}
		(\mathrm I)_{\mathrm a}&\leq  \E\bigg[\int_{\R^{2d}} \, \int_{B_{R}(x_1)}\,\int_{B_{R}(x_2)}\,\, \beta_n(\d z_1)\, \beta_n(\d z_2)\1\{|x_1-x_2| \geq 2R\}  \\
		               &\qquad\times \,\, \bE^{\otimes}_{\ssup{z_1,z_2}}\bigg[\1\{W^{\ssup 1}_t\in \d x_1\} \, \1\{W^{\ssup 2}_t\in \d x_2\} \,\, \1\{|W^{\ssup 1}_s- W^{\ssup 2}_s| >1 \,\,\forall \, s\in [0,t]\}\exp\bigg\{\beta \sum_{i=1}^2\mathscr H_t(W^{\ssup i},B)\bigg\}\bigg]\bigg] \\
		               &\qquad + \eta(R) 
		               \end{aligned}
		               \end{equation}
		               where $\eta(R)$ is defined canonically, and it is easy to see that for any fixed $t$ and uniformly in $n$, $\lim_{R\to\infty} \eta(R)=0$. Indeed, on the event $\{|W^{\ssup 1}_s - W^{\ssup 2}_s| \leq 1\,\,\,\mbox{for some}\,\,s\in [0,t]\}$, we have 
		               \begin{equation}\label{Ia-1}
		               \mathbf E\bigg[\exp\big\{\beta \sum_{i=1}^2\mathscr H_t(W^{\ssup i},B)\big\}\bigg] \leq \e^{2\beta^2 t V(0)} 		               
		               \end{equation}
		               and, therefore,
		               $$
		               \begin{aligned}
		               \eta(R)&= \e^{2\beta^2 t V(0)}\,\, \int_{\R^{2d}} \, \int_{B_{R}(x_1)}\,\int_{B_{R}(x_2)}\,\, \beta_n(\d z_1)\, \beta_n(\d z_2)\1\{|x_1-x_2| \geq 2R\}  \\
		               &\qquad\times \,\, \bP^{\otimes}_{\ssup{z_1,z_2}}\bigg[W^{\ssup 1}_t\in \d x_1, \, \,\, W^{\ssup 2}_t\in \d x_2 ,\,\, |W^{\ssup 1}_s- W^{\ssup 2}_s| \leq1\,\,\,\mbox{for some}\,\,\, s\in[0,t]\bigg] \\
		               		               \end{aligned}
		               $$
		               and the last probability is equal to $\bP_{z_1-z_2}\big\{\sqrt 2W_t\in \d (x_1-x_2), \,\,\, \sqrt 2 |W_s| \leq1\,\,\,\mbox{for some}\,\,\, s\in[0,t]\big\}$ whose integral on the set $|x_1-x_2|\geq 2R$ above vanishes as $R\to\infty$.

		               	We now focus on the first expectation on the right-hand side in \eqref{Ia}. Recall that $\phi$ has support in a ball of radius $1/2$ around the origin, and on the event $\{|W^{\ssup 1}_s - W^{\ssup 2}_s| > 1\,\,\,\mbox{for all}\,\,s\in [0,t]\}$, we have 	
			\begin{equation}\label{Ia-2}
			\mathbf E\bigg[\exp\big\{\beta\sum_{i=1}^2\mathscr H_t(W^{\ssup i},B)\big\}\bigg]  = \prod_{i=1}^2 \E\bigg[\e^{\beta\int_0^t \int_{\R^d} \phi(W^{\ssup i}_s-y) \dot B(s,\d y) \,\d s}\bigg]= \e^{\beta^2 t V(0)}.
			\end{equation}
			Hence, by \eqref{Ia}, 
			\begin{equation}\label{Ia-3}
			\begin{aligned}
			(\mathrm I)_{\mathrm a}&\leq  \e^{\beta^2 t V(0)}\,\, \int_{\R^{2d}} \, \int_{B_{R}(x_1)}\,\int_{B_{R}(x_2)}\,\, \beta_n(\d z_1)\, \beta_n(\d z_2)\1\{|x_1-x_2| \geq 2R\}  \\
		               &\qquad\times \,\, \bP^{\otimes}_{\ssup{z_1,z_2}}\bigg[W^{\ssup 1}_t\in \d x_1,\,\,W^{\ssup 2}_t\in \d x_2, \,\, |W^{\ssup 1}_s- W^{\ssup 2}_s| >1 \,\,\forall \, s\in [0,t]\bigg] \\
		               &\qquad + \eta(R).
		               \end{aligned}
		               \end{equation}	
		               
		               In order to conclude the proof of \eqref{claim-lemma-step5}, we now compute $\big(\E[\mathscr F_t(\beta_n)]\big)^2$ in a similar manner as \eqref{eq0-step5}. Since all the integrands are nonnegative, we can get a lower bound:
		               \begin{equation}\label{claim2-lemma-step5}
		               \begin{aligned}
		                \big(\E[\mathscr F_t(\beta_n)]\big)^2 &\geq 	\e^{\beta^2 t V(0)}\,\, \int_{\R^{2d}} \, \int_{B_{R}(x_1)}\,\int_{B_{R}(x_2)}\,\, \beta_n(\d z_1)\, \beta_n(\d z_2)\1\{|x_1-x_2| \geq 2R\}  \\
		               &\qquad\times \,\, \bP^{\otimes}_{\ssup{z_1,z_2}}\bigg[W^{\ssup 1}_t\in \d x_1,\,\,W^{\ssup 2}_t\in \d x_2, \,\, |W^{\ssup 1}_s- W^{\ssup 2}_s| >1 \,\,\forall \, s\in [0,t]\bigg].  
		               \end{aligned}
		               \end{equation}     
		               We combine \eqref{II}, \eqref{Ib}, \eqref{Ia-3} and \eqref{claim2-lemma-step5}, and first let $n\to\infty$, and then pass to $R\to\infty$ to complete the proof of \eqref{claim-lemma-step5}, and also of Proposition \ref{lemma-step5}.       \qed

		Finally, we will provide the following. 
		
		\noindent{\bf{Proof of Proposition \ref{prop-step3}:}} 
		
		Recall that if $\xi_n\to \xi$ in $(\X,\mathbf D)$, we want to show that, for any $p\geq 2$, 
		\begin{equation}\label{claim-prop-step3}
		\E\left|\frac{(\mathscr F_t(\xi)+\E[Z_t-\mathscr F_t(\xi)])^p-(\mathscr F_t(\xi_n)+\E[Z_t-\mathscr F_t(\xi_n)])^p}{(\mathscr F_t(\xi)+\E[Z_t-\mathscr F_t(\xi)])^p}\right|\rightarrow 0.
		\end{equation}
		We again recall the convergence criterion \eqref{decompose-1}-\eqref{decompose-3}. Also, note that, given any $\delta>0$, we can choose $k\in \N$ large enough such that, $\sum_{i>k} \alpha_i(\R^d) \leq \delta$ where $\Psi(\xi)=\sum_i \alpha_i(\R^d)$ and $\xi=(\widetilde\alpha_i)_i$. In order to prove \eqref{claim-prop-step3}, we first recall the notation 
		$$
		\overline{\mathscr F}_t(\xi_n)= \E[Z_t]+ \sum_{i=1}^k \big[\mathscr F_t(\alpha_{n,i})- \E[\mathscr F_t(\alpha_{n,i})]\big].
		$$
	  By the binomial theorem, 
		\begin{equation}\label{binomial}
		\begin{aligned}
		\big[\overline{\mathscr F}_t&(\xi_n)+ \mathscr F_t(\beta_n)- \E[\mathscr F_t(\beta_n)]\big]^p \\
		&=\overline{\mathscr F}_t(\xi_n)^p + [\mathscr F_t(\beta_n)- \E[\mathscr F_t(\beta_n)]] \, B_n		\end{aligned}
		\end{equation}
		where 
		\begin{equation}\label{B_n}
		 B_n= \sum_{l=0}^{p-1} \binom {p} l \, \big(\overline{\mathscr F}_t(\xi_n)\big)^l\,\, \big(\mathscr F_t(\beta_n)-\E[\mathscr F_t(\beta_n)]\big)^{p-1-l}.
		\end{equation}
		Then
		\begin{equation}\label{claim2-prop-step3}
		\begin{aligned}
		\mathrm{(L.H.S.) \,\,in \,\, \eqref{claim-prop-step3}} &\leq \delta^{\prime}+ \E\bigg[\bigg| \frac{\overline{\mathscr F}_t(\xi_n)^p -\overline{\mathscr F}_t(\xi)^p}{(\mathscr F_t(\xi)+\E[Z_t-\mathscr F_t(\xi)])^p}\bigg|\bigg] \\
		&\quad+ \E\bigg[\bigg| \frac{\mathscr F_t(\beta_n)- \E[\mathscr F_t(\beta_n)]}{(\mathscr F_t(\xi)+\E[Z_t-\mathscr F_t(\xi)])^p} \,\,  B_n\bigg|\bigg],
\end{aligned}
\end{equation}	
where $\delta^{\prime}\to 0$ as $\delta\to 0$. We will show that both expectations on the right-hand side above converge to $0$ as $n\to\infty$. First, for both terms we will invoke Cauchy-Schwarz bound again. For the first expectation, note that by Lemma \ref{lemma-step4},
$$
\E\bigg[\Big(\overline{\mathscr F}_t(\xi_n)^p -\overline{\mathscr F}_t(\xi)^p\Big)^2\bigg]	\to 0
$$
while, by the argument proving \eqref{eq-CS-bound}, we have, for a finite constant $C_1$
$$
\E\bigg[\big(\mathscr F_t(\xi)+\E[Z_t-\mathscr F_t(\xi)]\big)^{-2p}\bigg] \leq C_1.
$$
Now for the second expectation, we invoke Proposition \ref{lemma-step5} to get 
$$
\lim_{n\to\infty} \E\big[(\mathscr F_t(\beta_n)- \E[\mathscr F_t(\beta_n)])^2 \big]=0,
$$
while again by \eqref{eq-CS-bound} we have, for a finite constant $C_2$, 
$$
\E\bigg[\bigg(\mathscr F_t(\xi)+\E[Z_t-\mathscr F_t(\xi)]\big)^{-4p}\bigg] \leq C_2,
$$
and we claim that for another finite constant $C_3$, 
\begin{align}
\label{bound on Bn}
\sup_n \E[B_n^4] \leq C_3.
\end{align}
The last five assertions, together with successive application of the Cauchy-Schwarz inequality imply that both expectations on the right-hand side in \eqref{claim2-prop-step3} converge to $0$. Finally, we let $\delta\to 0$ to complete the proof of Proposition \ref{prop-step3}. 

{We owe the reader only the proof of \eqref{bound on Bn}. Using that $\big|\mathscr F_t(\alpha)-\E[\mathscr F_t(\alpha)]\big|\leq \mathscr F_t(\alpha)+\E[\mathscr F_t(\alpha)]$ and two more applications of the binomial theorem, yield
$$
B_n^4\leq p^4(\E[Z_t]+\mathscr F_t(\xi_n)+\E[\mathscr F_t(\xi_n)])^{4p-4}\leq p^4\sum_{l=0}^{4p-4}\binom{4p-4}{l}(\E[Z_t])^l(\mathscr F_t(\xi_n)+\E[\mathscr F_t(\xi_n)])^{4p-4-l}
$$
which together with $\E[\mathscr F_t(\xi_n)^k]\leq \Psi(\xi_n)^{k}\e^{k^2\beta^2tV(0)/2}\leq \e^{k^2\beta^2tV(0)/2}$ (recall \eqref{Jensen for upper bound}) proves the existence of $C_3$ in \eqref{bound on Bn}. }
\qed

We will end this section with a useful remark. 

\begin{remark}\label{rmk-Pi}
For any $\vartheta\in \Mcal_1(\X)$, let us set
\begin{equation}\label{Pi}
\Pi_t(\vartheta,\d\xi^\prime)= \int_{\X} \pi_t(\xi,\d\xi^\prime) \vartheta(\d\xi).
\end{equation}
Then by Theorem \ref{continuity of the pi-function}, for any $t>0$, the map $\vartheta\mapsto \Pi_t(\vartheta, \cdot)\in \Mcal_1(\X)$ is continuous. Furthermore,
since $\Pi_t(\delta_{\widetilde 0}, \cdot)= \P[\xi^{\ssup t} \in \cdot | \xi=\widetilde 0]=\delta_{\widetilde 0}$, the set
\begin{equation}\label{mathcal-K}
\mathfrak {m}:=\big\{\vartheta\in \mathcal{M}_1(\X):\Pi_t\vartheta=\vartheta\text{ for all } t>0\big\}
\end{equation}
of all fixed points of $\Pi_t$ is nonempty. 
Moreover, both $\X$ and, therefore, $\Mcal_1(\X)$ are compact, in their respective topologies. Thus, any sequence $\vartheta_n$  in $\mathfrak m$ has a subsequence that has a limit $\vartheta \in \Mcal_1(\X)$.  
The aforementioned continuity of $\vartheta\to \Pi_t(\vartheta,\cdot)$ guarantees that this limit $\vartheta\in \mathfrak m$, proving that $\mathfrak m$ is closed and, therefore, also compact. 
\qed
\end{remark}

\subsection{The free energy variational formula.} 

We now state the main result of this subsection, which provides a variational formula for the polymer free energy. Recall that functional $\Phi$ from \eqref{function Phi} and $\mathscr I_\Phi$ from \eqref{scr-I}.

\begin{theorem}
	\label{thm 4.9.}
	With $\mathfrak m$ defined in \eqref{mathcal-K}, 
	\begin{align}
		\label{(4.22)}
		\lim_{T\rightarrow\infty} \frac 1 T \, \E\left[{\log Z_T}\right]=\inf_{\vartheta\in\mathfrak{m}}\mathscr I_{\Phi}(\vartheta)
	\end{align}
	and 
	\begin{align}
		\label{(4.21)}
		\lim_{T\rightarrow\infty}\frac 1 T \, {\log Z_T}=\inf_{\vartheta\in\mathfrak{m}}\mathscr I_{\Phi}(\vartheta)\quad\text{a.s.}
	\end{align}
\end{theorem}

\begin{remark}
	Recall that the renormalized partition function $\mathscr Z_T$ is directly related to the SHE solution $u_\eps$ (recall \eqref{she-intro} and \eqref{scaling}). Likewise $\log \mathscr Z_T$ is related to the \textit{Cole-Hopf} solution $h_\eps:=\log u_\eps$, satisfying the Kardar-Parisi-Zhang (KPZ) equation
	$$
	\partial_t h_\eps=\frac{1}{2}\Delta h_\eps+\Big[\frac{1}{2}|\triangledown h_\eps|^2-C_\eps\Big]+\beta\eps^{\frac{d-2}{2}}\dot B_\eps
	$$
	with $h_\eps(0,x)=0$, see \cite{CCM18,CCM19,CCM19-II} for recent progress about the behavior of the limiting solution as $\eps\to 0$ in $d\geq 3$ and for small $\beta$. Since $\frac{1}{T}\log\mathscr Z_T=\frac{1}{T}\log Z_T-\frac{\beta^2}{2}V(0)$, Theorem \ref{thm 4.9.} then provides a law of large numbers for the KPZ solution $h_\eps$.\qed
\end{remark}
Theorem \ref{thm 4.9.} will follow from the following (almost sure) law of large numbers. 

\begin{theorem}\label{thm-LLN}
The occupation measures $\nu_T=\frac 1 T \int_0^T \delta_{\widetilde\bQ_t} \d t$ are attracted by the set $\mathfrak m$, that is, $\mathscr W(\nu_T, \mathfrak m) \to 0$ almost surely.
\end{theorem}

We defer the proof of Theorem \ref{thm-LLN} to Section \ref{sec-final-details} and {first prove the following statements which will be used in the proof of Theorem \eqref{thm 4.9.}. 
\begin{lemma}
	\label{lemma for Ito for scr F}
	For $\xi\in\X$, if $\overline{\mathscr F}_T(\xi)=\mathscr F_T(\xi)+\E[Z_T-\mathscr F_T(\xi)]$ with $\mathscr F_T(\xi)$ defined in \eqref{scr-F}, then
	\begin{equation}\label{Ito for scr F}
\begin{aligned}
	&\log(\overline{\mathscr F}_T(\xi))\\
&\qquad=\int_0^T\d t\,\int_{\R^d}\frac{\beta}{\overline{\mathscr F}_t(\xi)}\sum_{\widetilde{\alpha}\in\xi}\int_{\R^d}\alpha(\d z)\bE_z\Big[\phi(y-W_t)\e^{\beta\mathscr H_t(W,B)}\Big]\dot B(t,\d y)\\
&\qquad\qquad+\frac{\beta^2}{2}\bigg[V(0)
		-\sum_{\widetilde{\alpha}_1,\widetilde\alpha_2\in\xi}\int_{\R^{2d}}V(x_2-x_1)\prod_{j=1}^{2}\frac{1}{\overline{\mathscr F}_t(\xi)}\int_{\R^d}\alpha_j(\d z_j)\bE_{z_j}\big[\1_{\{W_t^{\ssup j}\in\d x_j\}}\e^{\beta\mathscr H_t(W,B)}\big]\bigg].
	\end{aligned}
	\end{equation}
	In particular,
	\begin{align}
	\begin{split}
	\label{expectation log scr F}
	\E\log(\overline{\mathscr F}_T(\xi))=\E\bigg[\int_0^T&\frac{\beta^2}{2}V(0)-\frac{\beta^2}{2}\sum_{\widetilde{\alpha}_1\in\xi}\sum_{\widetilde{\alpha}_2\in\xi}\int_{\R^{2d}}V(x_2-x_1)\\
	&\prod_{j=1}^{2}\frac{1}{\overline{\mathscr F}_t(\xi)}\int_{\R^d}\alpha_j(\d z_j)\bE_{z_j}\big[\1_{\{W_t^{\ssup j}\in\d x_j\}}\e^{\beta\mathscr H_t(W,B)}\big]\d t\bigg].
	\end{split}		
	\end{align}
\end{lemma} 
\begin{proof}
	To prove \eqref{Ito for scr F}, we proceed in the same way as in the proof of Lemma \ref{Rewrite of the partition function}. Recall that
	\begin{align*}
	\mathscr F_T(\xi)&=\sum_{\widetilde{\alpha}\in\xi}\int_{\R^d}\int_{\R^d}\alpha(\mathrm{d}z) \bE_z\bigg[\1_{\{W_T\in\mathrm{d}x\}}\,\, \exp\big\{\beta\mathscr H_T(W,B)\big\}\bigg]=\sum_{\widetilde{\alpha}\in\xi}\int_{\R^d}\alpha(\mathrm{d}z) \bE_z\bigg[\exp\big\{\beta\mathscr H_T(W,B)\big\}\bigg]
	\end{align*}
	and thus $\overline{\mathscr F}_T(\xi)=\sum_{\widetilde{\alpha}\in\xi}\int_{\R^d}\alpha(\mathrm{d}z) \bE_z\big[\exp\big\{\beta\mathscr H_T(W,B)\big\}\big]+(1-\Psi(\xi))\E Z_T$, where $\Psi(\xi)=\sum_i \alpha_i(\R^d)$ as before.  By It\^{o}'s formula, 
	\begin{align}
	\label{Ito to scr F -1}
		\d \overline{\mathscr F}_T(\xi)&=\beta\int_{\R^d}\sum_{\widetilde{\alpha}\in\xi}\int_{\R^d}\alpha(\d z)\bE_z\Big[\phi(y-W_T)\e^{\beta\mathscr H_T(W,B)}\Big]\dot B(T,\d y)\d T\\
		&+\frac{\beta^2}{2}\sum_{\widetilde{\alpha}\in\xi}\int_{\R^d}\alpha(\d z)\bE_z\Big[V(0)\e^{\beta\mathscr H_T(W,B)}\Big]\d T
		+(1-\Psi(\xi))\frac{\beta^2}{2}V(0)\e^{\frac{\beta^2}{2}TV(0)}\d T\\
		\label{Ito to scr F -1.5}
		&=\beta\int_{\R^d}\sum_{\widetilde{\alpha}\in\xi}\int_{\R^d}\alpha(\d z)\bE_z\Big[\phi(y-W_T)\e^{\beta\mathscr H_T(W,B)}\Big]\dot B(T,\d y)\d T+\frac{\beta^2}{2}V(0)\overline{\mathscr F}_T(\xi)\d T.
	\end{align}
	The quadratic variation of the above term is now given by
	\begin{align}
	\label{Ito to scr F -2}
	\!\!\!\!\!\!\!\!\!\!\!\d \langle\overline{\mathscr F}_T(\xi)\rangle=\beta^2\sum_{\widetilde{\alpha}_1\in\xi}\sum_{\widetilde{\alpha}_2\in\xi}\int_{\R^{2d}}\alpha_1(\d z_1)\alpha_2(\d z_2)\bE^{\otimes}_{(z_1,z_2)}\Big[V(W_T^{\ssup 1}-W_T^{\ssup 2})\e^{\beta(\mathscr H_T(W^{\ssup 1},B)+\mathscr H_T(W^{\ssup 2},B))}\Big]\d T
	\end{align}
	with $W^{\ssup 1}$ and $W^{\ssup 2}$ being two independent Brownian motions starting at $z_1$ and $z_2$, respectively. We now apply It\^{o}'s formula to $\log(\overline{\mathscr F}_T(\xi))$ and plug in \eqref{Ito to scr F -1}-\eqref{Ito to scr F -1.5} as well as \eqref{Ito to scr F -2} to get
	\begin{align*}
		\log(\overline{\mathscr F}_T(\xi))&=\int_0^T\frac{1}{\overline{\mathscr F}_t(\xi)}\d\overline{\mathscr F}_t(\xi)-\frac{1}{2}\int_0^T\frac{1}{\overline{\mathscr F}^2_t(\xi)}\d\langle\overline{\mathscr F}_t(\xi)\rangle\\
		&=\int_0^T\int_{\R^d}\frac{\beta}{\overline{\mathscr F}_t(\xi)}\sum_{\widetilde{\alpha}\in\xi}\int_{\R^d}\alpha(\d z)\bE_z\Big[\phi(y-W_t)\e^{\beta\mathscr H_t(W,B)}\Big]\dot B(t,\d y)+\frac{\beta^2}{2}V(0)\\
		&-\frac{\beta^2}{2}\sum_{\widetilde{\alpha}_1\in\xi}\sum_{\widetilde{\alpha}_2\in\xi}\int_{\R^{2d}}V(x_2-x_1)\prod_{j=1}^{2}\frac{1}{\overline{\mathscr F}_t(\xi)}\int_{\R^d}\alpha_j(\d z_j)\bE_{z_j}\big[\1_{\{W_t^{\ssup j}\in\d x_j\}}\e^{\beta\mathscr H_t(W,B)}\big]\d t,
	\end{align*}
	which proves \eqref{Ito for scr F} and, therefore, \eqref{expectation log scr F}.
\end{proof}
}	
	
\noindent{\bf{Proof of Theorem \ref{thm 4.9.} [Assuming Theorem \ref{thm-LLN}]:}}  {Note that by the definitions of $\mathscr I_{\Phi}$ and $\Pi_t$ (recall Remark \ref{rmk-Pi}), we have for any $t$
	\begin{align}
		\mathscr I_{\Phi}(\Pi_t\delta_{\xi})=\int_{\X}\Phi(\xi^{\prime})\Pi_t(\delta_{\xi},\mathrm{d}\xi^{\prime})=\int_{\X}\Phi(\xi^{\prime})\, \mathbf P\big[\xi^{\ssup t} \in \d\xi^\prime|\xi\big]=\E\big[\Phi(\xi^{\ssup t})\big]. 
	\end{align}
	On the other hand, $\Phi(\xi^{\ssup t})=\frac{\beta^2}{2}V(0)\big(1-\frac{1}{V(0)}\sum_{\widetilde{\alpha}\in\xi}\int_{\R^d\times\R^d}V(x_2-x_1)\prod_{j=1}^2\alpha^{\ssup t}(\mathrm{d}x_j)\big)$	and so
	\begin{align}
	\label{identity between I and Phi}
	\mathscr I_{\Phi}(\Pi_t\delta_{\xi})=\E\bigg[\frac{\beta^2}{2}V(0)\Big(1-\frac{1}{V(0)}\sum_{\widetilde{\alpha}\in\xi}\int_{\R^d\times\R^d}V(x_2-x_1)\prod_{j=1}^2\alpha^{\ssup t}(\mathrm{d}x_j)\Big)\bigg].
	\end{align}
	We claim that
	\begin{align}
		\label{inequality between scr I and log}
		\int_0^T\mathscr I_{\Phi}(\Pi_t\delta_{\xi})\d t\geq \E\big[\log(\mathscr F_T(\xi)+\E[Z_T-\mathscr F_T(\xi)])\big].
	\end{align}
	For proving \eqref{inequality between scr I and log}, we start by considering the sum on the right-hand side of \eqref{identity between I and Phi}. Since $V,\alpha,Z_t$ and $\mathscr F_t(\xi)$ are nonnegative,
	\begin{align*}
		\sum_{\widetilde{\alpha}\in\xi}&\int_{\R^d\times\R^d}V(x_2-x_1)\prod_{j=1}^2\alpha^{\ssup t}(\mathrm{d}x_j)\\
		&=\sum_{\widetilde{\alpha}\in\xi}\int_{\R^d\times\R^d}V(x_2-x_1)\prod_{j=1}^2\frac{1}{\mathscr F_t(\xi)+\E[Z_t-\mathscr F_t(\xi)]}\alpha(\d z_j)\bE_{z_j}\big[\1_{\{W_t^{\ssup j}\in \d x_j\}}\e^{\beta\mathscr H_t(W,B)}\big]\\
		&\leq \sum_{\widetilde{\alpha}_1\in\xi}\sum_{\widetilde{\alpha}_2\in\xi}\int_{\R^d\times\R^d}V(x_2-x_1)\prod_{j=1}^2\frac{1}{\mathscr F_t(\xi)+\E[Z_t-\mathscr F_t(\xi)]}\alpha_j(\d z_j)\bE_{z_j}\big[\1_{\{W_t^{\ssup j}\in \d x_j\}}\e^{\beta\mathscr H_t(W,B)}\big],
	\end{align*}
	thus by \eqref{identity between I and Phi},
	\begin{align*}
		\mathscr I_{\Phi}(\Pi_t\delta_{\xi})\geq \E\Big[\frac{\beta^2}{2}&V(0)-\frac{\beta^2}{2}\sum_{\widetilde{\alpha}_1\in\xi}\sum_{\widetilde{\alpha}_2\in\xi}\int_{\R^d\times\R^d}V(x_2-x_1)\\
		&\prod_{j=1}^2\frac{1}{\mathscr F_t(\xi)+\E[Z_t-\mathscr F_t(\xi)]}\alpha_j(\d z_j)\bE_{z_j}\big[\1_{\{W_t^{\ssup j}\in \d x_j\}}\e^{\beta\mathscr H_t(W,B)}\big]\Big].
	\end{align*}
	The claim in \eqref{inequality between scr I and log} now follows from Lemma \ref{lemma for Ito for scr F}. We restrict to the case, where the total mass functional satisfies $\Psi(\xi)>0$ and we use the concavity of the logarithm, which implies that
	\begin{align}
	\begin{split}
	\label{using concavity}
	\E \big[\log(\mathscr F_T(\xi)+\E[Z_T-\mathscr F_T(\xi)])\big]&=\E\Big[\log\big(\Psi(\xi)\frac{\mathscr F_T(\xi)}{\Psi(\xi)}+(1-\Psi(\xi))\E Z_T\big)\Big]\\
	&\geq \Psi(\xi)\E\log\Big(\frac{\mathscr F_T(\xi)}{\Psi(\xi)}\Big)+(1-\Psi(\xi))\log\big(\E Z_T\big).
	\end{split}
	\end{align}
	As $\int\Psi(\xi)^{-1}\sum_{\widetilde{\alpha}\in\xi}\alpha(\d x)=1$, we can use Jensen's inequality, so that
	\begin{align*}
	\log\Big(\frac{\mathscr F_T(\xi)}{\Psi(\xi)}\Big)=\log\bigg(\int_{\R^d}\bigg(\frac{\sum_{\widetilde{\alpha}\in\xi}\alpha(\d z)}{\Psi(\xi)}\bigg)\bE_z\Big[\e^{\beta\mathscr H_T(W,B)}\Big]\bigg)
	\geq \int_{\R^d}\bigg(\frac{\sum_{\widetilde{\alpha}\in\xi}\alpha(\d z)}{\Psi(\xi)}\bigg)\log\bigg(\bE_z\Big[\e^{\beta\mathscr H_T(W,B)}\Big]\bigg)
	\end{align*}
	and since $\bE_z\big[\e^{\beta\mathscr H_T(W,B)}\big]\overset{(d)}{=}Z_T$,
	$$
	\E\log\Big(\frac{\mathscr F_T(\xi)}{\Psi(\xi)}\Big)\geq \int_{\R^d}\bigg(\frac{\sum_{\widetilde{\alpha}\in\xi}\alpha(\d z)}{\Psi(\xi)}\bigg)\E\log Z_T=\E\log Z_T.
	$$
	By using Jensen's inequality once more, $\log\E Z_T\geq \E\log Z_T$, and both lower bounds, together with \eqref{inequality between scr I and log} and \eqref{using concavity}, yield  $\int_0^T \mathscr I_\Phi(\Pi_t \delta_\xi) \, \d t\geq \E[\log Z_T]$ for any $\xi\in \X$. The last inequality, when $\Psi(\xi)=0$, follows immediately by Jensen's inequality. Indeed, if $\Psi(\xi)=0$, $\mathscr I_\Phi(\Pi_t \delta_\xi)=\frac{\beta^2}{2}V(0)$ for all $t$ and so $\int_0^T \mathscr I_\Phi(\Pi_t \delta_\xi) \, \d t=\log \E Z_T\geq \E[\log Z_T]$. Since $\int_0^T \mathscr I_\Phi(\Pi_t \delta_\xi) \, \d t\geq \E[\log Z_T]$ now holds unconditionally,} 
    for any $\vartheta\in \mathfrak m$, 
        $$
        \begin{aligned}
        \frac 1 T\E[\log Z_T] \leq \frac 1 T\int_{\X} \vartheta(\d\xi)\,\, \int_0^T \mathscr I_\Phi(\Pi_t \delta_\xi) \, \d t &= \frac 1 T \int_0^T \d t \int_{\X} \vartheta(\d\xi) \, \mathscr I_\Phi(\Pi_t \delta_\xi) \\
        &= \frac 1 T \int_0^T \d t \mathscr I_\Phi(\Pi_t\vartheta) 
        = \mathscr I_\Phi(\vartheta)
        \end{aligned}
        $$
        proving that, $\limsup_{T\to\infty}\frac 1 T\E[\log Z_T] \leq \inf_{\vartheta\in\mathfrak m} \mathscr I_\Phi(\vartheta)$. To prove the corresponding lower bound, note that by Corollary \ref{cor-I-Phi}, 
        $\liminf_{T\to\infty} \frac 1 T \log Z_T= \liminf_{T\to\infty} \mathscr I_\Phi(\nu_T)$ almost surely. Now Theorem \ref{thm-LLN} dictates $\mathscr W(\nu_T,\mathfrak m)\to 0$ almost surely and we know 
        that $\mathscr I_\Phi(\cdot)$ is continuous. Therefore, 
        \begin{equation}\label{assertion-4-thm4.9.}
        \liminf_{T\to\infty}\frac 1 T \log Z_T = \liminf_{T\to\infty} \mathscr I_\Phi(\nu_T)\geq \inf_{\vartheta\in \mathfrak m} \mathscr I_\Phi(\vartheta) \qquad\mbox{a.s.}
        \end{equation}
         On the other hand, again by Corollary \ref{cor-I-Phi}, $\frac 1 T\E[\log Z_T]= \E[\mathscr I_\Phi(\nu_T)]$. Since both $\Phi$ and $\mathscr I_\Phi$ are nonnegative, by Fatou's lemma and \eqref{assertion-4-thm4.9.}, 
        $$
        \liminf_{T\to\infty}\frac 1 T\E[\log Z_T] = \liminf_{T\to\infty} \E[\mathscr I_\Phi(\nu_T)] \geq \E\big[\liminf_{T\to\infty} \mathscr I_\Phi(\nu_T)] \geq  \inf_{\vartheta\in\mathfrak m} \mathscr I_\Phi(\vartheta)
        $$
        and, therefore, $ \lim_{T\to\infty}\frac 1 T\E[\log Z_T]   =\inf_{\vartheta\in\mathfrak m} \mathscr I_\Phi(\vartheta)$. Finally, we apply Theorem \ref{theoremF} with any $\delta\in (0,1)$ to 
        conclude 
	$$\lim_{T\rightarrow\infty}\frac{\log Z_T}{T}=\lim_{T\rightarrow\infty}\frac{\E\log Z_T}{T}= \inf_{\vartheta\in\mathfrak {m}}\mathscr I_{\Phi}(\vartheta)\quad\text{a.s.}$$
	\qed


\section{Final details}\label{sec-final-details}
We will now conclude the proof of Theorem \ref{APA} in this section. Given the results of Section \ref{section time dependence} and \ref{section Space X}, the arguments appearing 
in this part will closely follow the approach of \cite{BC16} adapted to our setting modulo slight modifications. In order to keep the present material self-contained, we will spell out the technical details.

\subsection{Proof of Theorem \ref{thm-LLN}.} 
In this section, we will complete the proof of Theorem \ref{thm-LLN} for which we will need a technical fact. 

Recall that we denote by $\mathrm{Lip}_1(0)$ the space of all Lipschitz functions $\ell : \X\to \R$ vanishing at $\widetilde 0$ and having Lipschitz constant $\leq 1$.  Then, for any fixed $\ell \in  \mathrm{Lip}_1(0)$, 
and $s\geq 0$, we set

\begin{equation}\label{Upsilon-T}
\begin{aligned}
\Theta_T(\ell)=\int_0^T\theta_t(\ell)\mathrm{d}t
\qquad\mbox{where}\qquad
\theta_t(\ell)=\ell(\widetilde\bQ_{t+s})-\E\big[ \ell(\widetilde\bQ_{t+s}) |\mathcal{F}_t\big].
\end{aligned}
\end{equation} 
The next lemma asserts that for any fixed $\ell$, $\Theta_T(\ell)$ has a sublinear growth at infinity. 
\begin{lemma}
	\label{Proposition for almost surely convergence}
	For any $\ell\in \mathrm{Lip}_1(0)$, 
	\begin{align}
		\label{Upsilon converges a.s.}
		\lim_{T\rightarrow\infty}\frac{|\Theta_T(l)|}{T}=0\quad\text{a.s.}
	\end{align}
\end{lemma}
\begin{proof}
We claim that for any $\ell\in \mathrm{Lip}_1(0)$, $s\geq 0$ and $n\in \N$, there exists a constant $C=C(\ell,s)\in (0,\infty)$ such that 
\begin{equation}\label{eq-star}
\E\big[\Theta_{s n} (\ell) ^4\big] \leq C n^2.
\end{equation}
The above estimate implies that $\sum_{n=1}^\infty \P\big[ \frac{\Theta_{sn}(\ell)}{sn} \geq (sn)^{-1/5}\big] \leq  C^\prime  \sum_{n=1}^\infty  n^{-6/5} <\infty$. Then \eqref{Upsilon converges a.s.} follows at once since 
with $n=\lfloor \frac{T}s\rfloor$ we have $\frac{\Theta_T(\ell)}T =\frac{\Theta_{sn}(\ell)}T + \frac 1 T\int_{sn}^T \theta_t(\ell) \, \d t$. The first term converges almost surely to $0$ by Borel-Cantelli lemma, while the second term  is bounded above by $2s/T$, since $|\theta_t(\ell)|\leq 2$. 

It remains to check \eqref{eq-star}. Note that for any $t \in [0,s)$, $M_{n,t}(\ell)= \sum_{k=0}^n \theta_{t+ks}(\ell)$ is an $(\mathcal F_{t+ (n+1)s})_{n\in\N_0}$ martingale and $\Theta_{s n} (\ell)= \int_0^s M_{n-1,t}(\ell) \, \d t$. Then by the Burkholder-Davis-Gundy inequality, $\E[ M_{n,t}(\ell)^4] \leq C (n+1)^2$ and subsequently, by Jensen's inequality, 
$\E\big[\big(\int_0^s M_{n-1,t}(\ell) \, \d t\big)^4\big] \leq  C n^2$, which proves \eqref{eq-star}.
\end{proof}

We will now conclude the following.

\noindent{\bf{Proof of Theorem \ref{thm-LLN}:}} For any fixed $s\geq 0$, we set 
\begin{equation}\label{nu-s}
\nu_T^{\ssup s}=\frac{1}{T}\int_0^T\delta_{\widetilde\bQ_{t+s}}\mathrm{d}t
\end{equation}
and recall from \eqref{Wasserstein-dual} the dual representation of the Wasserstein metric $\mathscr{W}(\vartheta,\vartheta^\prime)=\sup_{\ell\in\mathrm{Lip}_1(0)} \big| \int_{\X} \ell(\xi) \vartheta(\d\xi) - \int_{\X} \ell(\xi) \vartheta^\prime(\d\xi)\big|$ on $\Mcal_1(\X)$. Then for any $\ell\in \mathrm{Lip}_1(0)$, 
\begin{align*}
		\mathscr {W}(\nu_T,\nu_T^{\ssup s})
		=\sup_{\ell}\left(\frac{1}{T}\int_0^T \ell(\widetilde\bQ_t)\mathrm{d}t-\frac{1}{T}\int_s^{T+s}\ell(\widetilde\bQ_t)\mathrm{d}t\right)
		&=\sup_{\ell}\left(\frac{1}{T}\int_0^s\ell(\widetilde\bQ_t)\mathrm{d}t-\frac{1}{T}\int_0^{s}\ell(\widetilde\bQ_{T+t})\mathrm{d}t\right)\\
		&\leq \frac{1}{T}2s, 
	\end{align*} 
 and Theorem \ref{thm-LLN} follows once we show that, for any fixed $s\geq 0$, 
\begin{equation}\label{nu-s-K}
\mathscr {W}(\nu_T^{\ssup s},\Pi_s\nu_T)\rightarrow 0.
\end{equation}	
Recall \eqref{Upsilon-T} and note that
	$$
	\mathscr{W}(\nu_T^{\ssup s},\Pi_s\nu_T)=\sup_{\ell}\frac{\Theta_T(l)}{T}.
	$$
	By the definition of the metric $\mathbf D$ on $\X$, for any $\ell\in \mathrm{Lip}_1(0)$, 
	$\sup_{\xi\in \X}|\ell(\xi)|\leq\sup_{\xi\in\X}\mathbf{D}(\xi,\widetilde 0)\leq 2$ and 
	thus, the family of functions $\ell\in \mathrm{Lip}_1(0)$  is equicontinuous and closed in the uniform norm. By Ascoli's theorem, this space is then compact and is also separable. 	
	Lemma \ref{Proposition for almost surely convergence} guarantees 
	\begin{align}
		\label{(4.4)}
		\lim_{T\rightarrow\infty}\frac{\Theta_T(\ell_n)}{T}=0\quad\text{for all }n\geq 1
	\end{align}
	for any countable dense set $(\ell_n)_n$. Further, given any $\ell_1,\ell_2\in\mathrm{Lip}_1(0)$ with $\|\ell_1-\ell_2\|_{\infty}<\delta$, we have
	$|\Theta_T(\ell_1)- \Theta_T(\ell_2)|<  2\delta T$ by \eqref{Upsilon-T}. 
	Thus $(\Theta_T(\cdot)/T)_{T\geq 0}$ is equicontinuous on the compact metric space $\mathrm{Lip}_1(0)$, and since this family $\Theta_T/T$ converges pointwise to $0$ on a dense subset $(\ell_n)_n$, and 
	again by the Ascoli's theorem this convergence is uniform. Thus,  $\mathscr {W}(\nu_T^{\ssup s},\Pi_s\nu_T)\rightarrow 0$ as $T\to\infty$, which proves \eqref{nu-s-K}, and thus also Theorem \ref{thm-LLN}. 
	\qed

We will now deduce a corollary to Theorem \ref{thm-LLN}. Let us set  
\begin{equation}\label{eq-m-0}
\mathfrak {m}_0=\big\{\vartheta_0\in\mathfrak {m}:\mathscr I_{\Phi}(\vartheta_0)=\inf_{\vartheta\in\mathfrak {m}}\mathscr I_{\Phi}(\vartheta)\big\},
\end{equation}
with $\mathscr I_\Phi$ defined in \eqref{scr-I}. Again the continuity of $\vartheta\to \mathscr I_\Phi(\vartheta)$ guarantees compactness of $\mathfrak m_0$ (recall Remark \ref{rmk-Pi}).

 \begin{cor}
	\label{convergence to M}
	The measure $\nu_T$ converges in the Wasserstein metric to $\mathfrak m_0$ for $T\rightarrow\infty$.
\end{cor} 
\begin{proof}
The proof is a straightforward application of the triangle inequality combined with the preceding results. Indeed, by Corollary \ref{cor-I-Phi}, $|\mathscr I_\Phi(\nu_T)- \frac 1 T \log Z_T| \to 0$ almost surely, 
while Theorem \ref{thm 4.9.} dictates $|\frac 1 T\log Z_T- \inf_{\mathfrak m} \mathscr I_\Phi| \to 0$ almost surely. Therefore, $\mathscr I_\Phi(\nu_T)$ can be made arbitrarily close to $\inf_{\mathfrak m} \mathscr I_\Phi$ for $T$ large enough. Combining 
this statement with the fact that $\mathscr W(\nu_T,\mathfrak m)\to 0$ (from Theorem \ref{thm-LLN}), continuity of the functional $\mathscr I_\Phi(\cdot)$ (from Lemma \ref{lower semi cont for Phi}), compactness of $\mathfrak m$ (from Remark \ref{rmk-Pi}) and triangle inequality proves the desired claim.
\end{proof}

\subsection{Proof of Theorem \ref{APA}}
\label{section Proof}

In this section, we will conclude the proof of Theorem \ref{APA}, which involves two main steps. 

\noindent{\bf{Step 1:}}  
 With the compact set $\mathfrak m_0\subset \mathfrak m$ defined in \eqref{eq-m-0}, the first step shows that in the very strong disorder regime, under any $\vartheta\in \mathfrak m_0$  there is no disintegration of mass. 
\begin{theorem}
	\label{NoDisint}
	If $\beta$ is large enough such that $\Lambda(\beta)>0$ (see Theorem \ref{Lyapunov}), then $\vartheta\big[\xi\in\X:\Psi(\xi)=1\big]=1$ for any $\vartheta\in\mathfrak m_0$, where $\Psi(\xi)=\sum_i \alpha_i(\R^d)$ is the total mass functional on $\X$. 
\end{theorem}
For the proof of this theorem, we use the following lemma.
\begin{lemma}
	\label{Proposition 4.4}
	If $\vartheta\in\mathfrak m$, then $\vartheta\big[\xi\in\X:\Psi(\xi)=0\big] + \vartheta\big[\xi\in\X:\Psi(\xi)=1\big]=1$. 
\end{lemma}
\begin{proof}
	Suppose $\xi\in \X$ such that $\Psi(\xi)\in (0,1)$. Recall the definition of $\xi^{\ssup t}=\{\alpha_i^{\ssup t}\}_i$, from \eqref{alpha after time t}, and note that $\E[Z_t] \geq \E[\mathscr F_t(\xi)]$. Then applying Jensen's inequality to the strictly concave function $x\mapsto \frac x {x+ \E[Z_t- \mathscr F_t(\xi)]}$ we have for any $t>0$,

\begin{align}
		\E \big[\Psi(\xi^{\ssup t})\big]&=\E\bigg[\frac{\sum_{i \in I}\int_{\R^d}\int_{\R^d}\alpha_i(\mathrm{d}z)\bE_z\big[\1_{\{W_t\in\mathrm{d}x\}}\,\e^{\beta\int_0^t\int_{\R^d}\phi(y-W_s)\dot{B}(s,y)\d y \mathrm{d}s}\big]}{\mathscr F_t(\xi)+\E[Z_t-\mathscr F_t(\xi)]}\bigg]\nonumber\\
		&<\frac{\E[\mathscr F_t(\xi)]}{\E[\mathscr F_t(\xi)]+\E[Z_t-\mathscr F_t(\xi)]}\label{strict}\\
		&=\sum_{i \in I}\int_{\R^d}\int_{\R^d}\alpha_i(\mathrm{d}z)\bP_z(W_t\in\mathrm{d}x)\nonumber
		=\sum_i\alpha_i(\R^d)= \Psi(\xi).\nonumber
	\end{align}
	We remark that the inequality \eqref{strict} is strict because of strict concavity and nondegeneracy of $\P$. 
	Now let $\vartheta\in \mathfrak m\subset \Mcal_1(\X)$ be such that $\vartheta\big[ \xi \colon \Psi(\xi)\in (0,1)\big]>0$. Then by the strict upper bound \eqref{strict}, for any $t>0$, 
	$
	\int \Psi(\xi^\prime) \, \Pi_t(\vartheta, \d\xi^\prime)= \int \vartheta(\d\xi) \, \E[\Psi(\xi^{\ssup t})]  < \int \vartheta(\d\xi) \, \Psi(\xi)
	$, and since $\Pi_t \vartheta= \vartheta$ for any $t\geq 0$, we have a contradiction. To complete the proof of the lemma, note that for any $\xi\in \X$ with $\Psi(\xi)=0$ implies $\Psi(\xi^{\ssup t})=0$ and $\xi\in \X$ with $\Psi(\xi)=1$ implies $\Psi(\xi^{\ssup t})=1$.
	
	
\end{proof}
We will now provide the proof of Theorem \ref{NoDisint}.\\

\noindent{\bf{Proof of Theorem \ref{NoDisint}:}} Recall that $\delta_{\widetilde 0} \in \mathfrak m$. Suppose $\mathfrak m=\{\delta_{\widetilde 0}\}$. Then by  Theorem \ref{thm 4.9.}, 
	$$\lim_{T\rightarrow\infty}\E\left[\frac{\log Z_T}{T}\right]= \mathscr I_{\Phi}(\delta_{\widetilde 0})=\frac{\beta^2}{2}V(0)=\frac{\log\left(\E Z_T\right)}{T}$$
	which implies that  $\Lambda(\beta)=0$ (recall \eqref{beta_1}). But our assumption $\beta >\beta_1=\inf\{\beta>0\colon \Lambda(\beta)>0\}$ provides a contradiction. 
	Hence, there exists $\vartheta\in \mathfrak m$ such that $\vartheta\neq \delta_{\widetilde 0}$. 
	Lemma \ref{Proposition 4.4} guarantees that $\vartheta(B)>0$ with $B=\{\xi\in \X\colon \Psi(\xi)=1\}$. We will show that if $\vartheta(B)<1$, then $\vartheta \notin \mathfrak m_0$. 
	
	Note that $\xi^{\ssup t} \in B$ if and only if $\xi \in B$, and hence for any $A\subset \X$, 

	$$
	\pi_t(\xi,A)=\pi_t(\xi,A\cap B)\quad\text{for }\xi\in B \qquad\mbox{\and}\quad
	\pi_t(\xi,A\cap B)=0\quad\text{for}\,\, \xi\notin B.
	$$
	Using these two identities and with $\vartheta(\cdot | B)$ denoting the conditional probability on $\X$, 
	\begin{align*}
		\Pi_t(\vartheta(\cdot|B),A)=\int_{\X}\pi_t(\xi,A)\vartheta(\mathrm{d}\xi|B) 
		&=\frac{1}{\vartheta(B)}\int_B\pi_t(\xi,A)\vartheta(\mathrm{d}\xi)\\
		&=\frac{1}{\vartheta(B)}\left(\int_{B}\pi_t(\xi,A\cap B)\vartheta(\mathrm{d}\xi)+\int_{B^{\mathsf{c}}}\pi_t(\xi,A\cap B)\vartheta(\mathrm{d}\xi)\right)\\
		&=\frac{1}{\vartheta(B)}\int_{\X}\pi_t(\xi,A\cap B)\vartheta(\mathrm{d}\xi)\\
		&=\frac{1}{\vartheta(B)}\Pi_t(\vartheta,A\cap B).
	\end{align*}	
	Hence,  $\vartheta\in\mathfrak m$ implies $\vartheta (\cdot|B) \in\mathfrak m$. Let us assume that $\vartheta(B)<1$. Then we will show that
	$\mathscr I_{\Phi}[\vartheta(\cdot |B)]<\mathscr I_{\Phi}(\vartheta)$, which in turn would imply that $\vartheta\notin\mathfrak{m}_0$ giving us a contradiction. 
	
	Recall that the map $\Phi$ is continuous and $\Phi(\widetilde 0)= \beta^2 V(0)/2$. Then if $\xi\neq \widetilde 0$, 
	$$ \Phi(\xi)=\frac{\beta^2}{2}V(0)\left(1-\frac{1}{V(0)}\sum_{i\in I}\int_{\R^d\times\R^d}V(x_2-x_1)\prod_{j=1}^2\alpha_i(\mathrm{d}x_j)\right)<\frac{\beta^2}{2}V(0)=\Phi(\widetilde 0)$$
	and hence, 
	\begin{align}
		\mathscr I_{\Phi}[\vartheta(\cdot |B)]=\frac{1}{\vartheta(B)}\int_{B} \Phi(\xi)\vartheta(\mathrm{d}\xi)
		&=\int_{B} \Phi(\xi)\vartheta(\mathrm{d}\xi)+\frac{1-\vartheta(B)}{\vartheta(B)}\int_{B} \Phi(\xi)\vartheta(\mathrm{d}\xi)\nonumber\\
		&<\int_{B} \Phi(\xi)\vartheta(\mathrm{d}\xi)+(1-\vartheta(B))\Phi(\widetilde 0)\nonumber\\
		&=\int_{B} \Phi(\xi)\vartheta(\mathrm{d}\xi)+\int_{B^{\mathsf{c}}} \Phi(\xi)\vartheta(\mathrm{d}\xi)\label{Lemma 4.4}\\
		&=\mathscr I_{\Phi}(\vartheta)\nonumber
	\end{align}
	and we used Lemma \ref{Proposition 4.4} in the identity \eqref{Lemma 4.4}. We conclude that $\vartheta$ can only be an element of $\mathfrak {m}_0$, if $\vartheta(B)=1$.
\qed

\medskip

\noindent{\bf{Step 2:}} We will now conclude the following.

\noindent{\bf{Proof of  Theorem \ref{APA}}}: Recall from Section \ref{subsec-total-mass} the  functional $\Psi_\eps$ on $\X$ and the associated lower semicontinuous integral functional $\mathscr I_{\Psi_\eps}(\vartheta)=\int \Psi_\eps(\vartheta) \vartheta(\d\xi)$ on $\Mcal_1(\X)$. 
	For any $\xi\in\X$ with $\Psi(\xi)=1$, $\Psi_{\eps}(\xi)\nearrow 1$ for $\eps\rightarrow 0$. Since we assume $\Lambda(\beta)>0$, Theorem \ref{NoDisint} and monotone convergence theorem imply that
	$$\mathscr I_{\Psi_{\eps}}(\vartheta)\nearrow 1$$ pointwise for any $\vartheta\in \mathfrak m_0$. 
	Since $\mathfrak m_0$ is compact this pointwise convergence is in fact uniform. Thus, for any $m \in (0,1)$, there exists $\eps>0$  such that $\mathscr I_{\Psi_{\eps}}(\vartheta)>m$ for all $\vartheta\in\mathfrak {m}_0$.
	By compactness of $\mathcal{M}_1(\X)$ and lower semicontinuity of $\mathscr I_{\Psi_{\eps}}$, for any such $m\in (0,1)$ and $\eps>0$, we can find $\delta>0$ such that for any $\mu\in \Mcal_1(\X)$, 
	$\mathscr {W}(\mu,\mathfrak {m}_0)<\delta$ implies $\mathscr I_{\Psi_{\eps}}(\mu)>m$. 
	Thus, for any given $m\in (0,1)$ we can choose $\eps>0$ and $\delta>0$ such that $\mathscr I_{\Psi_{\eps}}(\vartheta)>m$ for all $\vartheta\in\mathfrak{m}_0$ and $\mathscr I_{\Psi_{\eps}}(\mu)>m$ for $\mu\in \mathcal{M}_1(\X)$ and so by Corollary \ref{convergence to M} there is a.s. $T^{\ast}$ large enough that
	$$T\geq T^{\ast}\Rightarrow \mathscr {W}(\nu_T,\mathfrak m_0)<\delta\Rightarrow \mathscr I_{\Psi_{\eps}}(\nu_T)>m.
	$$
	Now if we recall the relation \eqref{APA-psi-eps}, we have shown that, given any $m\in (0,1)$, there exists $\eps>0$ such that 
	$$
	\liminf_{T\rightarrow\infty}\frac{1}{T}\int_0^T\hP_t(W_t\in U_{t,\eps})\mathrm{d}t>m\quad\text{a.s.}
	$$
	
Note that the last display also implies the proof of Theorem \ref{APA}. Indeed, given any $k\in \N$, assume that the last assertion holds for some $\eps>0$ and $m\in(\tau , 1)$ with $\tau= 1- \frac 1 k$. Then we choose $T_1, T_2, T_3$ such that for $T>T_1$, we have $\int_0^T \hP_t(W_t\in U_{t,\eps})\mathrm{d}t>m T$, 
and for $t\geq T_2$ we have $\eps_t<\eps$ and so $\hP_t(W_t\in U_{t,\eps})\geq \hP_t(W_t\in U_{t,\eps})$, and for $T_3>T_2$ we have $m-\frac{T_2}{T_3}>\tau$. Now we conclude for $T\geq \max\{T_1,T_3\}$:
	\begin{align*}
	\frac{1}{T}\int_0^T\hP_t(W_t\in U_{t,\eps_t})\mathrm{d}t \geq\frac{1}{T}\int_{T_2}^T\hP_t(W_t\in U_{t,\eps_t})\mathrm{d}t
	&\geq\frac{T_2}{T}+\frac{1}{T}\int_{T_2}^T\hP_t(W_t\in {U}_{t,\eps_t} )\mathrm{d}t-\frac{T_2}{T_3}\\
	&\geq \frac{1}{T}\int_0^{T_2}\hP_t(W_t\in {U}_{t,\eps_t})\mathrm{d}t+\frac{1}{T}\int_{T_2}^T\hP_t(W_t\in U_{t,\eps_t})\mathrm{d}t-\frac{T_2}{T_3}\\
	&\geq \frac{1}{T}\int_0^{T}\hP_t(W_t\in {U}_{t,\eps_t})\mathrm{d}t-\frac{T_2}{T_3}\\
	&> m-\frac{T_2}{T_3}>\tau,
	\end{align*}
	which completes the proof of Theorem \ref{APA}. 


\appendix

\section{}\label{sec-appendix} 

Recall that in the proof of Theorem \ref{APA}, we needed $\Lambda(\beta)= -\lim_{T\to\infty} \frac 1 T \E[\log \mathscr Z_{\beta,T}]>0$ to use Theorem \ref{NoDisint}. 
The following monotonicity result for $\Lambda(\beta)$ was originally derived in \cite{CY06} for discrete directed polymers. 

\begin{theorem}
	\label{Lyapunov}
	The Lyapunov exponent 
	$$
	\Lambda(\beta)= -\lim_{T\to\infty} \frac 1 T \E[\log \mathscr Z_{\beta,T}]
	$$
	exists and is nonnegative. Furthermore, the map $\beta\mapsto \Lambda(\beta)$ is nondecreasing and continuous in $(0,\infty)$ and $\Lambda(0)=0$. Finally, $\Lambda(\beta)>0$ implies that 
	$\lim_{T\to\infty} \mathscr Z_{\beta,T}=0$ almost surely. 
\end{theorem}

\begin{proof}
The existence of the Lyapunov exponent is a consequence of a subadditivity argument (see \cite[Proposition 1.4]{CH02}), and the nonnegativity follows from a direct application of Jensen's inequality.

We first want to show that 

\begin{equation}\label{claim-appendix}
\E\bigg[\frac{\partial}{\partial\beta}\log\mathscr Z_{\beta,T}\bigg]\leq 0\qquad\mbox{ for all }\,\,\,\, \beta\in (0,\infty).
\end{equation}
 Therefore, fix $\beta^\ast\in(0,\infty)$ and set $I=[0,\beta^\ast]$. We apply Jensen's inequality to
get $\E[\sup_{\beta\in I} \mathscr Z_{\beta,T}^{-2}] <\infty$. Next, recall the GMC measure $\mathscr M_{\beta,T}$ from \eqref{eq:M} and note that 
$$
\frac{\partial}{\partial\beta} \mathscr Z_{\beta,T} = \bE_0\bigg[\bigg(\mathscr H_T(W,B)- \beta T V(0)\bigg) \frac {\d\mathscr M_{\beta,T}}{\d\bP_0}\bigg]. 
$$
We again apply Jensen's inequality followed by the Cauchy-Schwarz inequality 
to get 
$$
\E\bigg[\sup_{\beta\in I} \bigg(\frac{\partial \mathscr Z_{\beta,T}}{\partial\beta}\bigg)^2\bigg]<\infty.
$$
 Then we can use the Cauchy-Schwarz inequality once more to show 
	$$
	\E\bigg|\frac{\partial\log \mathscr Z_{\beta,T}}{\partial\beta}\bigg|=\E\bigg|\frac{1}{\mathscr Z_{\beta,T}}\frac{\partial\mathscr Z_{\beta,T}}{\partial \beta}\bigg|\leq\bigg(\E\bigg|\mathscr Z_{\beta,T}\bigg|^{-2}\E\bigg|\frac{\partial\mathscr Z_{\beta,T}}{\partial\beta}\bigg|^2\bigg)^{1/2},
	$$
	and thus, $\sup_{\beta\in I} \frac{\partial\log \mathscr Z_{\beta,T}}{\partial\beta} \in L^1(\P)$. Then we can conclude
	\begin{align}
		\label{integral-derivative-change}
		\frac{\partial}{\partial\beta^\ast}\E[\log\mathscr Z_{\beta^\ast,T}]=\frac{\partial}{\partial\beta^\ast}\E[\log\mathscr Z_{0,T}]+\frac{\partial}{\partial\beta^\ast}\E\bigg[\int_0^{\beta^\ast}\frac{\partial\log\mathscr Z_{\beta,T}}{\partial\beta}\mathrm{d}\beta\bigg]&=\frac{\partial}{\partial\beta^\ast}\int_0^{\beta^\ast}\E\Big[\frac{\partial\log \mathscr Z_{\beta,T}}{\partial\beta}\Big]\mathrm{d}\beta\nonumber\\
		&=\E\Big[\frac{\partial}{\partial\beta^\ast}\log\mathscr Z_{\beta^\ast,T}\Big]
	\end{align}
	for all $\beta^\ast\in(0,\infty)$. We will use \eqref{integral-derivative-change} to show \eqref{claim-appendix}.

	Note that for any fixed $T, \beta$ and $W$, the maps $\dot B \mapsto \mathscr H_T(W,B)- \beta T V(0)$ and $\dot B \mapsto - {\mathscr Z_{\beta,T}}^{-1}$ are
	 nondecreasing (see \cite{B05}) and since the law $\P$ of the noise $\dot B$ is a product measure,
	 we use the FKG-inequality applied to the tilted measure $\frac{\mathrm{d}\mathscr M_{\beta,T}}{\mathrm{d}\bP_0}\mathrm{d}\P$  
	 to obtain 
	 \begin{equation}\label{FKG}
	 \E\bigg[-\frac{\partial}{\partial\beta}\log\mathscr Z_{\beta,T}\bigg]\geq\bE_0\bigg[\E\bigg[-\frac{1}{\mathscr Z_{\beta,T}}\frac{\mathrm{d}\mathscr M_{\beta,T}}{\mathrm{d}\bP_0}\bigg]\E\bigg[\frac{\partial}{\partial\beta}\frac{\mathrm{d}\mathscr M_{\beta,T}}{\mathrm{d}\bP_0}\bigg]\bigg].
	 \end{equation}
	  By calculations similar to \eqref{integral-derivative-change}, $\E[\frac{\partial}{\partial\beta}\frac{\mathrm{d}\mathscr M_{\beta,T}}{\mathrm{d}\bP_0}]=\frac{\partial}{\partial\beta}\E[\frac{\mathrm{d}\mathscr M_{\beta,T}}{\mathrm{d}\bP_0}]=0$, which combined with \eqref{FKG} then implies \eqref{claim-appendix} and the desired monotonicity of $\Lambda(\beta)$. The continuity of $\beta\mapsto \Lambda(\beta)$ on $(0,\infty)$ is an immediate 
	  consequence of its convexity which follows from H\"older's inequality. 
	
	Finally, to show that $\Lambda(\beta)>0$ implies $\lim_{T\rightarrow\infty}\mathscr Z_{\beta,T}=0$ $\P$-a.s., note that 
	$\mathcal{V}:=\{\mathscr Z_{\beta,T}\not\to_{T\to \infty} 0\}$  is a tail event for the process
$t\to \dot B(t,\cdot)$ and, therefore, 
$\P(\mathcal{V})\in \{0,1\}$. 
So if $\lim_{T\rightarrow\infty}\mathscr Z_{\beta,T}>0$ almost surely, since for $x>0$, $-\log(x)<\infty$,

	$$\Lambda(\beta)=\lim_{T\rightarrow\infty}\frac{1}{T}\E\left[-\log\mathscr Z_{\beta,T}\right]\leq 0,$$
	which provides a contradiction. 
\end{proof}

\begin{theorem}\label{theoremF}
For any $d\geq 1$, $\beta>0$ and $\delta>0$, as $T\to\infty$,
\begin{equation}
        \log Z_T-\E[\log Z_T]=O\big(T^{\frac{1+\delta}{2}}\big)\qquad\P-\text{a.s.}
        \end{equation}
        \end{theorem}
        \begin{proof}
        This result  has been shown for a Poissonian environment in \cite[Theorem 2.4.1(b) and Corollary 2.4.2]{CY05}. The proof in our setting is a straightforward adaptation of this result modulo minor changes.
In particular, in the proof of \cite[Theorem 2.4.1(b)]{CY05}, 
the function $\varphi(v)=\e^v-v-1$  has to be replaced by $\varphi(v)=\frac{1}{2}v^2-v$, while the indicator function there has to be replaced by our fixed  mollifier $\phi$,  
and  the constant $C$ should be chosen to be $C=|B_{1/2}(0)| \big(\e^{\beta\|\phi\|_\infty}-1\big)^2$.
 \end{proof}

 \noindent{\bf{Acknowledgement.}} It is a pleasure to thank Sourav Chatterjee (Stanford) and Francis Comets (Paris) for their encouragement and inspiration. The authors are grateful to 
 Erik Bates (Stanford),  Sourav Chatterjee, Francis Comets and Vincent Vargas (Paris) for their valuable comments on an earlier draft, and to Yuri Bakhtin (NYU) and Donghyun Seo (NYU) for pointing out an error in the proof of 
Theorem 3.1 in the earlier version.


\end{document}